\pdfoutput=1
\RequirePackage{ifpdf}
\ifpdf 
\documentclass[pdftex]{sigma}
\else
\documentclass{sigma}
\fi
\numberwithin{equation}{section}

\newtheorem{thm}{Theorem}[section]
\newtheorem{cor}[thm]{Corollary}
\newtheorem{Lemma}[thm]{Lemma}
\newtheorem{prop}[thm]{Proposition}

{\theoremstyle{definition}
\newtheorem{Remark}[thm]{Remark}
}

\usepackage{amscd, tikz, tikz-cd}
\usepackage{enumerate,stmaryrd}
\usepackage[all,cmtip]{xy}
\usetikzlibrary{decorations.markings}
\tikzstyle{vertex}=[circle, draw, inner sep=0pt, minimum size=6pt]

\newcommand{\C}{{\mathbf C}}
\newcommand{\Z}{{\mathbf Z}}

\newcommand{\Ext}{\operatorname{Ext}}

\newcommand{\End}{\operatorname{End}}
\newcommand{\Hom}{\operatorname{Hom}}
\newcommand{\Ind}{\operatorname{Ind}}
\newcommand{\Res}{\operatorname{Res}}
\newcommand{\Ker}{\operatorname{Ker}}
\newcommand{\sdim}{\operatorname{sdim}}

\newcommand{\0}{\bar 0}
\newcommand{\gl}{\ensuremath{\mathfrak{gl}}}
\newcommand{\g}{\ensuremath{\mathfrak{g}}}
\newcommand{\p}{\ensuremath{\mathfrak{p}}}
\newcommand{\h}{\ensuremath{\mathfrak{h}}}

\newcommand{\q}{\ensuremath{\mathfrak{q}}}
\newcommand{\sq}{\ensuremath{\mathfrak{sq}}}

\newcommand{\ep}{\varepsilon}
\newcommand{\im}{\operatorname{im}}

\begin{document}

\allowdisplaybreaks

\newcommand{\arXivNumber}{2008.10649}

\renewcommand{\thefootnote}{}

\renewcommand{\PaperNumber}{141}

\FirstPageHeading

\ShortArticleName{Extension Quiver for Lie Superalgebra $\mathfrak{q}(3)$}

\ArticleName{Extension Quiver for Lie Superalgebra $\boldsymbol{\mathfrak{q}(3)}$\footnote{This paper is a~contribution to the Special Issue on Representation Theory and Integrable Systems in honor of Vitaly Tarasov on the 60th birthday and Alexander Varchenko on the 70th birthday. The full collection is available at \href{https://www.emis.de/journals/SIGMA/Tarasov-Varchenko.html}{https://www.emis.de/journals/SIGMA/Tarasov-Varchenko.html}}}

\Author{Nikolay GRANTCHAROV~$^\dag$ and Vera SERGANOVA~$^\ddag$}

\AuthorNameForHeading{N.~Grantcharov and V.~Serganova}

\Address{$^\dag$~Department of Mathematics, University of Chicago, Chicago, IL 60637, USA}
\EmailD{\href{mailto:nikolayg@uchicago.edu}{nikolayg@uchicago.edu}}

\Address{$^\ddag$~Department of Mathematics, University of California at Berkeley, Berkeley, CA~94720, USA}
\EmailD{\href{mailto:serganov@math.berkeley.edu}{serganov@math.berkeley.edu}}

\ArticleDates{Received August 31, 2020, in final form December 10, 2020; Published online December 21, 2020}

\Abstract{We describe all blocks of the category of finite-dimensional $\mathfrak{q}(3)$-supermodules by providing their extension quivers. We also obtain two general results about the representation of $\mathfrak{q}(n)$: we show that the Ext quiver of the standard block of $\q(n)$ is obtained from the principal block of $\mathfrak{q}(n-1)$ by identifying certain vertices of the quiver and prove a ``virtual'' BGG-reciprocity for $\mathfrak{q}(n)$. The latter result is used to compute the radical filtrations of $\mathfrak{q}(3)$ projective covers.}


\Classification{17B55; 17B10}

\renewcommand{\thefootnote}{\arabic{footnote}}
\setcounter{footnote}{0}

\section{Introduction}
The ``queer'' Lie superalgebra $\q(n)$ is an interesting super analogue of the Lie algebra $\mathfrak{gl}(n)$. Other related queer-type Lie superalgebras include the subsuperalgebra $\mathfrak{sq}(n)$ obtained by taking odd trace 0, and for $n\geq3$, the simple Lie superalgebra $\mathfrak{psq}(n)$ obtained by taking the quotient of the commutator $[\q(n),\q(n)]$ by the center. These queer superalgebras have a rich representation theory, partly due to the Cartan subsuperalgebra $\h$ not being abelian and hence having nontrivial representations, called Clifford modules.

Finite-dimensional representation theory of $\q(n)$ was initiated in \cite{Kac} and developed in \cite{P}. Algorithms for computing characters of irreducible finite-dimensional representations were obtained in \cite{PS2, PS1} using methods of supergeometry and in \cite{B1,B2} using a categorification approach. Finite-dimensional representations of half-integer weights were studied in detail in \cite{BD, CK,CKW}. In~\cite{Maz}, the blocks in the category of finite-dimensional
$\q(2)$-modules semisimple over the even part were classified and described using quivers and relations. A general classification of blocks was obtained in \cite{Ser-ICM} using translation functors and supergeometry.

In this paper, we describe the blocks in the category of finite-dimensional $\q(3)$ and $\mathfrak{sq}(3)$ modules semisimple over the even part in terms of quiver and relations. We found that to describe blocks of $\q(n)$ in general, it remains to consider the principal block. For $n=3$, this is the first example of a wild block in $\q$.
Our main tools are relative Lie superalgebra cohomology and geometric induction.

In Section~\ref{section2}, we describe some background information for $\q(n)$ and quivers, and we formulate our main theorems, Theorems~\ref{MainTheorem sq} and~\ref{MainTheorem}. In Section~\ref{section3}, we introduce geometric induction and prove a ``virtual'' BGG reciprocity law, Theorem \ref{BGG Reciprocity}, that generalizes~\cite{GS2} to the queer Lie superalgebras. This result allows us to describe radical filtrations of all finite-dimensional indecomposable projective modules for $\mathfrak{sq}(3)$ and $\q(3)$. Diagrams of these are provided in Appendix~\ref{appendixA}. In Section~\ref{section4}, we prove a result on self extensions of simples for $\g=\mathfrak{q}(n)$, Theorem~\ref{selfextensions}, and for $\g=\mathfrak{sq}(n)$, Theorem~\ref{selfextensions2}. In Section~\ref{section5}, we show the standard block for $\q(n)$ is closely related to the principal block of $\q(n-1)$, Proposition~\ref{tadmissible}, and in particular deduce the quiver for $\mathfrak{sq}(3)$ and $\q(3)$ standard block. Finally in Section~\ref{section6}, we compute the quiver for principal block of $\mathfrak{sq}(3)$ and $\q(3)$.

\section{Preliminaries and main theorem}\label{section2}

\subsection{General definitions}\label{section2.1} Throughout we work with $\C$ as the ground field. We set $\Z_2=\Z/2\Z$. Recall that a \textit{vector superspace} $V=V_{\bar{0}}\oplus V_{\bar 1}$ is a $\Z_2$-graded vector space. Elements of $V_{\overline{0}}$ and $V_{\overline{1}}$ are called even and odd, respectively. If $V$, $V'$ are superspaces, then the space $\Hom_{\C}(V,V')$ is naturally $\Z_2$-graded with grading $f\in\Hom_{\C}(V,V')_{ s}$ if $f(V_{{r}})\subset V_{{r+s}}'$ for all ${r}\in\Z_2$.

A \textit{superalgebra} is a $\Z_2$-graded, unital, associative algebra $A=A_{\overline{0}}\oplus A_{\overline{1}}$ which satisfies \mbox{$A_{{r}}A_{ s}\subset A_{{r+s}}$}. A \textit{Lie superalgebra} is a superspace $\g=\g_{\0}\oplus\g_{\bar{1}}$ with bracket operation $[\,,\, ]\colon \g\otimes\g\allowbreak \rightarrow \g$ which preserves the graded version of the usual Lie bracket axioms. The \textit{universal enveloping algebra} $U(\g)$ is $\Z_2$-graded and satisfies a PBW type theorem~\cite{Kac}.
A \textit{$\g$-module} is a left $\Z_2$-graded $U(\g)$-module. A \textit{morphism} of $\g$-modules $M\rightarrow M'$ is an element of $\Hom_{\C}(M,M')_{\bar 0}$ satisfying $f(xm)=xf(m)$ for all $m\in M,x\in U(\g)$. We denote by $\g$-mod the category of $\g$-modules. This is a symmetric monoidal category. The primary category of interest $\mathcal{F}$ consists of finite-dimensional $\g$-modules which are semisimple over $\g_{\0}$. We stress that we only allow for \textit{parity preserving} morphisms in $\mathcal{F}$. In this way, $\mathcal{F}$ is an abelian rigid symmetric monoidal category: for $V,W\in\mathcal{F}$, define $V\otimes W$ and $V^*$ using the coproduct and antipode of $U(\g)$, respectively:
\[ \delta(x)=x\otimes 1+1\otimes x,\qquad S(x)=-x\qquad \forall\, x\in\g.\]

For $V\in\g$-mod, we denote by $S(V)$ the symmetric superalgebra. As a $\g_{\0}$-module, $S(V)$ is isomorphic to $S(V)=S(V_{\bar{0}})\otimes\Lambda(V_{\bar{1}})$, where $\Lambda(V_{\bar{1}})$ is the exterior algebra of $V_{\bar{1}}$ in the category of vector spaces. For $V$ a $\g_1$-module and $W$ a $\g_2$-module, we define the outer tensor product $V\boxtimes W$ to be the $\g_1\oplus\g_2$-module with the action for $(q_1,q_2)\in \g_1\oplus\g_2$ given by
\[ (q_1,q_2)(v\boxtimes w):= (-1)^{\overline{q_2}\overline{v}}(q_1v\boxtimes q_2w).\]

We define the (super)dimension of $V\in\g$-mod as follows. Let $\C[\ep]$ be polynomial algebra with variable $\ep$ and denote two-dimensional $\C$-algebra $\C[\ep]/\big(\ep^2-1\big)$ as $\widetilde{\C}$. Then
\[\dim(V):=\dim_{\C}(V_{\overline{0}})+\dim_{\C}(V_{\bar{1}})\ep\in\widetilde{\C}.\]
The \textit{parity change functor} $\Pi\colon A\text{\rm-smod}\rightarrow A{\rm - smod}$ is defined as follows: For $M\in A\text{\rm-smod}$, $\Pi(M)_{\bar{0}}:=M_{\bar{1}}$ and $\Pi(M)_{\bar{1}}:=M_{\bar {0}}$ and the action on $m\in\Pi(M)$ is $a\cdot m = (-1)^{\bar{a}}am$. Lastly, if $f\colon M\rightarrow N$ is a morphism of supermodules, then $\Pi f\colon \Pi M\rightarrow\Pi N$ is $\Pi f=f$.

\subsection[The queer Lie superalgebra q(n)]{The queer Lie superalgebra $\boldsymbol{\q(n)}$}\label{section2.2}

By definition, the \textit{queer Lie superalgebra $\q(n)$} is the Lie subsuperalgebra of $\gl(n|n)$ leaving invariant an odd automorphism of the standard representation $p$ with the property $p^2=-1$. In matrix form,
\[\q(n)=\left\{\begin{pmatrix}
A & B\\
B & A
\end{pmatrix}\colon A, B \in \mathfrak{gl}_n(\C) \right\}, \qquad \text{if} \quad p = \begin{pmatrix}
0 & 1_n\\
-1_n & 0
\end{pmatrix}.
\] Let $\g=\q(n)$. The even (resp.\ odd) subspace of $\g$ consists of block matrices with $B = 0$ (resp.\ $A=0$). For $1\leq i,j\leq n$, we define the standard basis elements as
\[ e_{i,j}^{\0}=\begin{pmatrix}
E_{i,j} & 0\\
0& E_{i,j}
\end{pmatrix}\in\g_{\0}\qquad \text{and} \qquad e_{i,j}^{\bar{1}}=\begin{pmatrix}
0& E_{i,j}\\
E_{i,j}& 0\\
\end{pmatrix}\in\g_{\bar{1}},
\]
where $E_{i,j}$ denote the elementary matrix.
Observe the odd trace $\operatorname{otr}\left(\begin{smallmatrix}
A & B\\
B & A\\
\end{smallmatrix}\right):=\operatorname{tr}(B)$ annihilates the commutator $[\mathfrak{q}(n),\mathfrak{q}(n)]$. Let
\[ \mathfrak{sq}(n)=\{X\in\q(n)\colon \operatorname{otr}(X)=0\}.\]
Furthermore, $\operatorname{otr}(XY)$ defines a nondegenerate $\g$-invariant odd bilinear form on~$\g$. In particular, we have an isomorphism $\mathfrak{q}(n)^*\cong \Pi\mathfrak{q}(n)$ of $\mathfrak{q}(n)$-modules.

All {Borel} Lie superalgebras $\mathfrak{b}\subset\g$ are conjugate to the ``standard'' Borel, i.e., block matrices where $A, B\in\gl(n)$ are upper triangular. The \textit{nilpotent subsuperalgebra} $\mathfrak{n}$ consists of block matrices where $A$, $B$ are strictly upper triangular.

In the standard basis, the supercommutator has the form
\[ [e_{ij}^{\sigma}, e_{kl}^{\tau}] = \delta_{jk}e_{il}^{\sigma+\tau}-(-1)^{\sigma\tau}\delta_{il}e_{kj}^{\sigma+\tau},\]
where $\sigma,\tau\in\Z_2$. The \textit{Cartan superalgebra} $\h$ has basis $e_{ii}^{\sigma}$ for $1\leq i\leq n$, $\sigma\in\Z_2$. The elements $H_i:= e_{ii}^{\0}$, $\overline{H}_i:=e_{ii}^{\bar{1}}$, $1\leq i \leq n$, form a basis for $\h_{\0}$, $\h_{\bar{1}}$, respectively. Let $\{\varepsilon_i\,|\, i=1,\dots,n\}\subset\h_{\bar0}^*$ denote the dual basis of~$\{H_i\}$. There is a root decomposition of $\g$ with respect to the Cartan subalgebra $\h$ given by
\[\g=\h\oplus\bigoplus_{\alpha\in\Phi}\g_{\alpha},\]
where $\Phi = \{\varepsilon_i-\varepsilon_j \,|\, 1\leq i\neq j\leq n\}$ is the same as the set of roots of $\mathfrak{gl}_n(\C)$. For a root $\alpha=\ep_i-\ep_j$ we have $\dim\g_{\alpha}=1+\ep$ because $\g_{\alpha} = \operatorname{span}\{e_{i,j}^{\sigma}\colon \sigma\in\Z_2\}$. The positive roots are $\Phi^+:=\{\ep_i-\ep_j\colon 1\leq i<j\leq n\}$. The simple roots are $\{\ep_i-\ep_{i+1}\colon 1\leq i\leq n-1\}$. The Weyl group for $\q(n)$ is $W=S_n$, the symmetric group on~$n$ letters.

By $\mathfrak{h}'$ we denote the Cartan subsuperalgebra of $\mathfrak{sq}(n)$.
 A \textit{weight} is by definition an element $\lambda\in\h_{\0}^*$ and we write it in the form $\lambda = (\lambda_1,\dots,\lambda_n)$ with respect to the standard basis $(\varepsilon_1,\dots,\varepsilon_n)$. We say $\lambda$ is integral if $\lambda_i\in\Z$ for all $1\leq i\leq n$. We say $(\lambda_1,\dots,\lambda_n)\in\h_{\0}^*$ is \textit{typical} if $\lambda_i+\lambda_j\neq 0$ for all $1\leq i\neq j\leq n$. We introduce partial ordering on $\h_{\0}^*$ via $\lambda\leq \mu$ if and only if $\mu-\lambda\in\mathbf{N}\Phi^+$. Finally, we define $\rho_0:=1/2\sum_{\alpha\in\Phi^+}\alpha$.

\subsection[Irreducible h and g-representations]{Irreducible $\boldsymbol{\h}$ and $\boldsymbol{\g}$-representations}\label{section2.3}

Following \cite[Proposition~1]{P}, we now define for each $\lambda\in\h_{\0}^*$ a simple $\h$-supermodule. Define an even superantisymmetric bilinear form $F_{\lambda}\colon \h_{\bar{1}}\times\h_{\bar{1}}\rightarrow \C$ as $F_{\lambda}(u,v): = \lambda([u,v])$. Let $K_{\lambda} = \Ker F_{\lambda}$ and $E_{\lambda} = \h_{\bar{1}}/K_{\lambda}$. The restriction of~$F_\lambda$ to $\mathfrak{h}'$ will be denoted by $F'_\lambda$ and we set
 $K'_{\lambda} = \Ker F'_{\lambda}$ and $E'_{\lambda} = \h'_{\bar{1}}/K'_{\lambda}$.

 \begin{Lemma}\label{dim Clifford}
 Let $\lambda=(\lambda_1,\dots,\lambda_n)\in\h_{\bar{0}}^*$.
 \begin{enumerate}\itemsep=0pt
 \item[$(a)$] If there exists $i$ such that $\lambda_i=0$, then
 \[\dim E_{\lambda}=\dim E'_{\lambda}=|\{i\colon \lambda_i\neq 0\}|.\]
 \item[$(b)$] If all $\lambda_i\neq 0$ and $\frac{1}{\lambda_1}+\dots+\frac{1}{\lambda_n}\neq 0$, then
\[\dim E'_{\lambda} = n-1,\qquad \dim E_{\lambda}=n.\]

 \item[$(c)$] If all $\lambda_i\neq 0$ and $\frac{1}{\lambda_1}+\dots+\frac{1}{\lambda_n}= 0$, then
\[\dim E'_{\lambda} = n-2,\qquad \dim E_{\lambda}=n.\]
 \end{enumerate}
 \end{Lemma}

\begin{proof}
It is straightforward that $K_\lambda$ is the span of $\bar H_i$ for all $i$ such that $\lambda_i\neq 0$. Hence
\[\dim E_{\lambda}=|\{i\colon \lambda_i\neq 0\}|.\]

 To compute $K'_\lambda$, consider the basis $\{\bar H_i-\bar H_{i+1}\,|\, i=1,\dots,n-1\}$ of $\h'_{\bar{1}}$. Then
\begin{align*}
K'_{\lambda}&=\{u\in\h'_{\bar{1}}\,|\, \lambda([u,\bar{H}_i-\bar H_{i+1}])=0\text{ for all } 1\leq i\leq n-1\}\\
&=\{ (u_1,\dots,u_n)\in\h_{\bar1}\,|\, u_1+\dots+u_n=0, u_i\lambda_i=u_{i+1}\lambda_{i+1}\text{ for all } 1\leq i\leq n-1\}.
\end{align*}
Suppose first without loss of generality $\lambda_1=\dots = \lambda_k=0$, where $k\geq 1$. This forces $u_{k+1}=\dots=u_{n}=0$ and $u_1+\dots+ u_k=0$, so $K'_\lambda$ has a basis
\[ \big\{ \bar H_1-\bar H_2,\dots,\bar H_{k-1}-\bar H_k\big\}\]
and $\dim K'_{\lambda} = k-1$.
Thus, $\dim E'_{\lambda} = \dim\h'_{\bar{1}}-\dim {K'_{\lambda}} = n-k$.

Next, suppose all $\lambda_i\neq 0$. Then similarly we compute
\begin{equation*}K'_{\lambda}=\begin{cases} \C\big(\frac{1}{\lambda_1},\frac{1}{\lambda_2},\frac{1}{\lambda_3},\dots,\frac{1}{\lambda_n}\big) &\text{if } \frac{1}{\lambda_1}+\dots+\frac{1}{\lambda_n}=0,
\\ 0 &\text{if } \frac{1}{\lambda_1}+\dots+\frac{1}{\lambda_n}\neq 0. \end{cases}\tag*{\qed}
\end{equation*}\renewcommand{\qed}{}
\end{proof}

Let $\dim E_{\lambda}=m>0$. On the vector superspace $E_{\lambda}$, $F_{\lambda}$ induces a nondegenerate bilinear form, also denoted $F_{\lambda}$. Let $\operatorname{Cliff}(\lambda)$ be the \textit{Clifford superalgebra} defined by $E_{\lambda}$ and $F_{\lambda}$. Then (1)~$\operatorname{Cliff}(\lambda)$ is isomorphic to Cliff$(m)$, the Clifford superalgebra with generators $e_1,\dots,e_m$ and relations $e_i^2=1$, (2)~$\dim\operatorname{Cliff}(\lambda) = 2^{m-1}(1+\varepsilon)$, and (3) the category $\operatorname{Cliff}(\lambda)$-mod is semisimple (e.g.,~\cite{Mein}).

If $m$ is odd, then there exists a unique simple $\operatorname{Cliff}(m)$-module, denoted by $v(m)$, which is invariant under parity change (this follows from existence of an odd automorphism). If $m$ is even, then there exists $2$ nonisomorphic simple $\operatorname{Cliff}(m)$-modules $v(m)$ and $\Pi v(m)$ which are swapped by the parity change functor. Using the surjective homomorphism $U(\h)\rightarrow\operatorname{Cliff}(\lambda)$ with kernel $(H_i-\lambda_i, K_{\lambda})$, we lift $v(m)$ to an $\h$-module which we denote by $v(\lambda)$. Lemma~\ref{dim Clifford} implies
\[\dim v(\lambda) = \dim(v(m)) = 2^{\left \lfloor{(m-1)/2}\right \rfloor }(1+\varepsilon),\]
where $\lfloor x\rfloor$ denotes the integer part of $x\in\mathbf{R}$. Furthermore, this construction provides a complete irredundant collection of all finite-dimensional simple $\h$-supermodules.

Next define the \textit{Verma module}
\[ M_{\g}(\lambda) := U(\g)\otimes_{U(\mathfrak{b})} v(\lambda),\]
where the action of $\mathfrak{n}^+$ on $v(\lambda)$ is trivial.

Let
\[\Lambda=\{\lambda=(\lambda_1,\dots,\lambda_n)\in \h_{\bar{0}}^*\colon \lambda_{i}-\lambda_{i+1}\in\Z \}.\]
The set of $\g$-dominant integral weights is
\[ \Lambda^+=\{\lambda=(\lambda_1,\dots,\lambda_n)\in \h_{\bar{0}}^*\colon \lambda_{i}-\lambda_{i+1}\in\Z_{\geq0} \ \text{and} \ \lambda_i=\lambda_j\Rightarrow \lambda_i=\lambda_j=0\}.\]
Below is the main theorem about irreducible $\g$-modules, first proven by V.~Kac.
\begin{thm}[\cite{Kac}]\label{classification simples}\quad
\begin{enumerate}\itemsep=0pt
\item[$1.$] For any weight $\lambda\in\h_{\bar{0}}^*$, $M_{\g}(\lambda)$ has a unique maximal submodule $N(\lambda)$, hence a unique simple quotient, $L_{\g}(\lambda)$.
\item[$2.$] For each finite-dimensional irreducible $\g$-module $V$, there exists a unique weight $\lambda\in\Lambda^+$ such that V is a homomorphic image of $M_{\g}(\lambda)$.
\item[$3.$] $L _{\g}(\lambda) := M_{\g}(\lambda)/N_{\g}(\lambda)$ is finite dimensional if and only if $\lambda\in\Lambda^+$.
\end{enumerate}
\end{thm}
We will often omit the subscript $\g$ in the notation for Verma, simple, and projective modules.

\subsection[The category F]{The category $\boldsymbol{\mathcal{F}}$}\label{section2.4}

Let $\g=\q(n)$. Denote by $\mathcal{F}^n$, or simply $\mathcal{F}$ the category consisting of finite-dimensional $\g$-supermodules semisimple over $\g_{\bar{0}}$ (so the center of $\g_{\bar{0}}$ acts semisimply), with morphisms being parity preserving. The full subcategory of $\mathcal{F}$ consisting of modules with integral weights is equivalent to the category of finite-dimensional $G$-modules, where $G$ is the algebraic supergroup with $\text{Lie}(G)=\g$ and $G_{\0}={\rm GL}(n)$.

Let $Z(\g)$ be the center of the universal enveloping algebra $U(\g)$. A central character is a~homomorphism $\chi\colon Z(U(\g))\rightarrow\C$. We say that a $\g$-module $M$ has central character $\chi$ if for any $z\in Z(\g)$, $m\in M$, there exists a positive integer $n$ such that $(z-\chi(z)\text{id})^n.m=0$. It is well known from linear algebra that any finite-dimensional indecomposable $\g$-module has a central character, hence $\mathcal{F}^n=\oplus \mathcal{F}^n_{\chi}$, where $\mathcal{F}^n_{\chi}$ is the subcategory of modules admitting central character~$\chi$. In the most cases $\mathcal{F}^n_\chi$ is indecomposable, i.e., a block in the category $\mathcal {F}^n$. The only exception is $\mathcal {F}^n_\chi$ for even $n$ and typical central character $\chi$. In this case
$\mathcal {F}^n_\chi$ is semisimple and has two non-isomorphic simple objects $L(\lambda)$ and $\Pi L(\lambda)$.

Similarly to the Lie algebra case, there is a canonical injective algebra homomorphism, the \textit{Harish-Chandra homomorphism} \cite{CW, Serg},
\[ \text{HC}\colon \ Z(\g)\hookrightarrow S(\h_{\bar{0}})^W.\]
Given any $\lambda\in\h_{\bar{0}}^*$, we define $\chi_{\lambda}\colon Z(\g)\rightarrow \C$ to be the unique homomorphism making
\[\begin{tikzcd}[column sep=small]
Z(\g) \arrow[hookrightarrow]{rr}{{\rm HC}}\arrow{rd}[swap]{\chi_{\lambda}} & & S(\h_{\bar0})^W\arrow{ld}{\lambda^W} \\
 & \C &
\end{tikzcd}
\]
commute, where $\lambda^W$ is the natural homomorphism induced by $\lambda\in\h_{\bar0}^*$. If $\chi=\chi_{\lambda}$ for some $\lambda$, we denote $\mathcal{F}_{\chi_{\lambda}}$ by $\mathcal{F}_{\lambda}$.
Given a central character $\chi_{\lambda}$ with $\lambda =(\lambda_1,\dots,\lambda_n)$, we define its \textit{weight} to be the formal sum
\[ \text{\rm wt}(\lambda):=\delta_{\lambda_1}+\dots +\delta_{\lambda_n},\]
where $\delta_i = -\delta_{-i}$ and $\delta_0=0$. A fundamental result by Sergeev \cite{Serg} implies:
\begin{thm}\label{blocks-central characters}
For $\lambda,\mu\in\h_{\bar0}^*$, $\chi_{\lambda}=\chi_{\mu}$ if and only if $\text{\rm wt}(\lambda)=\text{\rm wt}(\mu)$.
\end{thm}

The following classification theorem about blocks in $\mathcal{F}^3$ is important for us. It is an immediate consequence of \cite[Theorem~5.8]{Ser-ICM}.

\begin{thm} \label{Ser-blocks} $\lambda=(\lambda_1,\lambda_2,\lambda_3)\in\Lambda^+\cap\Z^3$ be a dominant integral weight and $|\lambda|$ be the number of non-zero coordinates in $\text{\rm wt}(\lambda)$.
\begin{itemize}\itemsep=0pt
\item $($the strongly typical block$)$ If $|\lambda|=3$, then $\mathcal F^3_\lambda$ is semisimple and contains one up to isomorphism simple module;
\item $($the typical block$)$ If $|\lambda|=2$, then $\mathcal F^3_\lambda$ is equivalent to the block $\mathcal{F}^1_{(0)}$ for $\q(1)$;
\item $($the standard block$)$ If $|\lambda|=1$, then $\mathcal{F}^3_{\lambda}$ is equivalent to $\mathcal{F}^3_{(1,0,0)}$;
\item $($the principal block$)$ If $|\lambda|=0$, then $\mathcal{F}^3_{\lambda}$ is equivalent to $\mathcal{F}^3_{(0,0,0)}$.\end{itemize}
\end{thm}
Furthermore, if $\lambda\in\Lambda^+$ but $\lambda\notin\Z^3$, then either $\lambda$ is typical, so $\lambda_i+\lambda_j\neq 0$ $\forall\, i,j$, or $\lambda$ has atypicality~1. In the former case the block is semisimple and has one up to isomorphism simple object. In the latter case all such blocks are equivalent to the ``half-standard'' block $\mathcal{F}_{(3/2,1/2,-1/2)}$ by \cite[Theorem~5.21]{BD}.

Finally, it is well known there are enough projective and injective objects in $\mathcal{F}$ \cite{Ser1}. Let $P_{\g}(\lambda)$ denote the projective cover of $L_{\g}(\lambda)$.

\subsection{Quivers}\label{section2.5} Let $\mathcal{F}$ be any abelian $\C$-linear category with enough projectives, finite-dimensional morphism spaces, and finite-length composition series for all objects. For us, $\mathcal{F}$ will be as in the previous subsection. The following properties are as stated in \cite[Section~1]{Ger}, which are just slight generalizations of results in \cite[Section~4.1]{Ben}.

An \textit{Ext-quiver} $Q$ for $\mathcal{F}$ is a directed graph with vertex set consisting of isomorphism classes of finite-dimensional simple objects of $\mathcal{F}$. In our case, the vertex set is $Q_0=\{L(\lambda),\Pi L(\lambda)\}$ for $\lambda\in\Lambda^+$. In particular, $Q_0$ is not $\Lambda^+$. The number of arrows between two objects $L,M\in Q_0$ will be $d_{L,M}:=\dim\Ext^1_{\mathcal{F}}(L,M)$. We define a $\C$-linear category $\C Q$ with objects being vertices~$Q_0$ and morphisms $\Hom_{\C Q}(\lambda,\mu)$ being space of formal linear combinations of paths between the two objects $\lambda$,~$\mu$. Composition of morphisms is concatenation of paths.

A \textit{system of relations} on $Q$ is a map $R$ which assigns a subspace $R(\lambda,\mu)\subset\Hom_{\C Q}(\lambda,\mu)$ to each pair of vertices $(\lambda,\mu)\in Q_0\times Q_0$ such that for any $\lambda,\mu,\nu\in Q_0$
\begin{gather*}
R(\nu,\mu)\circ\Hom_{\C Q}(\lambda,\nu) \subset R(\lambda,\mu)\qquad\text{and}\qquad
\Hom_{\C Q}(\nu,\mu)\circ R(\lambda,\nu) \subset R(\lambda,\mu).
\end{gather*}

A \textit{representation of $Q$} is a finite-dimensional vector space $V=\oplus_{\lambda\in Q_0}V_{\lambda}$ together with linear maps $\phi\colon V_{\lambda}\rightarrow V_{\mu}$ for every arrow $\phi\colon \lambda\rightarrow\mu$. Representations of $Q$ form an Abelian category denoted by $Q$-mod. Given quiver $Q$ and relations $R$, define the category $\C Q/R$ consisting of objects $\lambda\in Q_0$ and morphisms $\Hom_{\C Q/R}(\lambda,\mu):=\Hom_{\C Q}(\lambda,\mu)/R(\lambda,\mu)$. We then denote by $\C Q/R$-mod the full subcategory of $\C Q$-mod consisting of representations $V$ such that for any vertices~$\lambda$,~$\mu$, we have $\operatorname{Im}(R(\lambda,\mu)\rightarrow\Hom_{\C}(V_{\lambda},V_{\mu}))=0$.

The next proposition gives an explicit description of the relations of an Ext-quiver given the category $\mathcal{F}$, its spectroid $\mathcal{G}$, and its Ext-quiver $Q$. The \textit{spectroid} $\mathcal{G}$ is defined as the full subcategory of $\mathcal{F}$ consisting of objects which are indecomposable projectives. Let $\mathcal{G}^{{\rm op}}$ denote the opposite category: objects are that of $\mathcal{G}$ and morphisms are $\Hom_{\mathcal{G}^{{\rm op}}}(P(\lambda),P(\mu)):=\Hom_{\mathcal{G}}(P(\mu),P(\lambda))$. Let $\operatorname{rad}(P(\lambda),P(\mu))$ denote the set of all noninvertible morphisms from~$P(\lambda)$ to~$P(\mu)$. Since $P(\mu)$ is projective, such a morphism cannot be surjective and we thus conclude $\operatorname{rad}(P(\lambda),P(\mu))=\Hom_{\mathcal{F}}(P(\lambda),\operatorname{rad} P(\mu))$.
 Let $\operatorname{rad}^n(P(\lambda),P(\mu))$ be the subspace of $\operatorname{rad}(P(\lambda),P(\mu))$ consisting of sums of products of $n$ noninvertible maps between $P(\lambda)$ and $P(\mu)$. For $\lambda,\mu\in\Lambda^+$ we have a canonical isomorphism \cite[Lemma~1.2.1]{Ger}
\[\Ext^1_{\mathcal{F}}(L(\lambda),L(\mu))\cong \Hom_{\mathcal{F}}\big(P(\mu),\operatorname{rad} P(\lambda)/\operatorname{rad}^2P(\lambda)\big)^*.\]

\begin{prop} \label{relations}
Given category $\mathcal{F}$ with Ext-quiver $Q$ and spectroid $\mathcal{G}$, let $\mathcal{R}_{\lambda,\mu}$ denote the bijection from the $d_{\lambda,\mu}$ arrows of $\lambda$ to $\mu$ to the family $\{\phi^i_{\lambda,\mu}\}_{i=1}^{d_{\lambda,\mu}}$ of morphisms in $\operatorname{rad}(P(\mu), P(\lambda))$ that map onto a basis modulo $\operatorname{rad}^2(P(\mu),P(\lambda))$. Then there is a unique well-defined family of linear maps
\[\overline{\mathcal{R}}_{\lambda,\mu}\colon \ \Hom_{\mathbf{C}Q}(\lambda,\mu)\rightarrow\Hom_{\mathcal F}(P(\mu),P(\lambda)),\]
such that $\overline{\mathcal{R}}_{\lambda,\mu}(\phi_{\lambda,\mu}^i)=\mathcal{R}_{\lambda,\mu}(\phi_{\lambda,\mu}^i)$ and which is compatible with composition.

Moreover, the map
\[ R\colon \ (\lambda,\mu)\rightarrow\operatorname{Ker}\overline{\mathcal{R}}_{\lambda,\mu}\] is a system of relations on~$Q$ and the categories $\C Q/R$ and $\mathcal{G}^{{\rm op}}$ are equivalent.
\end{prop}

The system of relations is determined up to a choice of $R_{\lambda,\mu}$ which is not canonical in general. But, we may multiply the $\mathcal{R}_{\lambda,\mu}(\phi_{\lambda,\mu}^i)$ by nonzero scalars to make the relations ``look nice''. There is then an additional proposition, \cite[Proposition~1.2.2]{Ger}, which states $\mathcal{G}^{{\rm op}}$ is equivalent to~$\mathcal{F}$. This then implies the following important theorem of Ext-quivers we use.

\begin{thm}[{\cite[Theorem 1.4.1]{Ger}}]\label{quiver = representations}
Let $\mathcal{F}$ be as above, $Q$ its Ext-quiver, and $R$ be a system of relations as defined in Proposition~{\rm \ref{relations}}. Then there exists an equivalence of categories
\[\mathbf{e}\colon \ \mathcal{F}\xrightarrow{\sim} \mathbf CQ/R-{\rm mod}
\]
such that
\[ \mathbf{e}(M) = \bigoplus_{\lambda\in \Lambda^+}\Hom_{\mathcal{F}}(P(\lambda),M).\]
\end{thm}

\subsection{Main theorem}\label{section2.6} In the statement of the main theorems, we will provide the Ext-quivers of various blocks. The relations are given by labelling the $\dim\Ext^1_{\g}(L(\lambda),L(\mu))$ arrows between $L(\lambda), L(\mu)\in Q$ by $\alpha\in\Hom_{Q}(L(\lambda),L(\mu))$ which is then identified (by some choice of scalar) with $\alpha\in\Hom_{\g}(P(\lambda),\\ \operatorname{rad} P(\mu)/\operatorname{rad}^2P(\mu))$ via Proposition~\ref{relations}.

\begin{thm}\label{MainTheorem sq}
Every block $\mathcal{F}_{\lambda}$ of the category $\mathcal{F}$ of finite-dimensional $\mathfrak{sq}(3)$-modules semisimple over $\mathfrak{sq}(3)_{\bar{0}}$ is equivalent to the category of finite-dimensional modules over one of the following algebras given by a quiver and relations:
\begin{enumerate}\itemsep=0pt
\item[$1.$] A typical block $\lambda = (\lambda_1,\lambda_2,\lambda_3)$ such that $\lambda_i+\lambda_j\neq 0$ for any $i\neq j$, and $\frac{1}{\lambda_1}+\frac{1}{\lambda_2}+\frac{1}{\lambda_3}\neq 0$ or exactly one $\lambda_i=0$
\[\xymatrix{\bullet .}
\]

\item[$2.$] A strongly typical block $\lambda = (\lambda_1,\lambda_2,\lambda_3)$ such that $\lambda_i+\lambda_j\neq 0$, $\lambda_i\neq0$ for any $i$, $j$ and $\frac{1}{\lambda_1}+\frac{1}{\lambda_2}+\frac{1}{\lambda_3}=0$
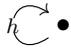
\[\xymatrix{\bullet\ar@(dl,ul)|{h} }
\]
with relations
\begin{equation*}h^2=0.\end{equation*}
\item[$3.$] The ``half-standard'' block $\lambda = (\frac{3}{2},\frac{1}{2},-\frac{1}{2})$
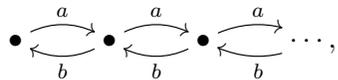
\[
\xymatrix{\bullet \ar@/^/[r]^{a}& \bullet
\ar@/^/[l]^{b}
\ar@/^/[r]^{a}& \bullet\ar@/^/[l]^{b}
\ar@/^/[r]^{a}&\cdots\ar@/^/[l]^{b},}
\]
where vertices are labeled $L_{\sq}\big(\frac{3}{2},\frac{1}{2},-\frac{1}{2}\big)$, $L_{\sq}\big(\frac{5}{2},\frac{3}{2},-\frac{5}{2}\big)$, $L_{\sq}\big(\frac{7}{2},\frac{3}{2},-\frac{7}{2}\big)$, $\dots$
with relations
 \begin{equation*}a^2=b^2=0, \qquad ab=ba.\end{equation*}
\item[$4.$] The standard block $\lambda = (1,0,0)$
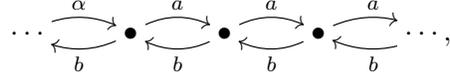
\[
\xymatrix{\cdots \ar@/^/[r]^{\alpha}&\bullet \ar@/^/[l]^{b} \ar@/^/[r]^{a}& \bullet
\ar@/^/[l]^{b}
\ar@/^/[r]^{a}& \bullet\ar@/^/[l]^{b}
\ar@/^/[r]^{a}&\cdots\ar@/^/[l]^{b},}
\]
where vertices are labeled $\dots$, $\Pi L_{\sq}(3,1,-3)$, $\Pi L_{\sq}(2,1,-2)$, $L_{\sq}(1,0,0)$, $L_{\sq}(2,1,-2)$, $L_{\sq}(3,1,-3)$, $\dots$
with relations
\begin{equation*}a^2=b^2=0,\qquad ab=ba.\end{equation*}

\item[$5.$] The principal block $\lambda = (0,0,0)$
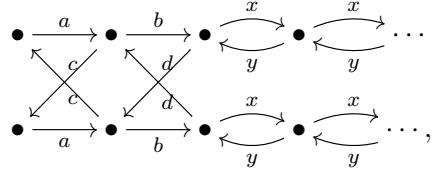
\[
\xymatrix{
\bullet \ar[r]^{a} & \bullet \ar[dl]^(.55){c} \ar[r]^{b} & \bullet\ar@/^/[r]^x \ar[dl]^(.55){d} &\bullet \ar@/^/[l]^y \ar@/^/[r]^x&\cdots \ar@/^/[l]^y\\
\bullet \ar[r]_{a} & \bullet \ar[r]_{b} \ar[ul]_(.55){c} & \bullet\ar[ul]_(.55){d}\ar@/^/[r]^x &\bullet \ar@/^/[l]^y \ar@/^/[r]^x& \cdots\ar@/^/[l]^y,
}\]
where vertices are labeled $L_{\sq}(1,0,-1)$, $L_{\sq}(0)$, $L_{\sq}(2,0,-2)$, $L_{\sq}(3,0,-3)$, $\dots$ in top row and $\Pi L_{\sq}(1,0,-1)$, $\Pi L_{\sq}(0,0,0)$, $\Pi L_{\sq}(2,0,-2)$, $\Pi L_{\sq}(3,0,-3)$, $\dots$ in bottom row. Then the relations are
\begin{gather*}
 x^2=y^2=0, \qquad xb= dy=bd= ca=0, \\ xy=yx,\qquad yx=bacd, \qquad
dbac=acdb. \end{gather*}
\end{enumerate}
\end{thm}

\begin{thm}\label{MainTheorem}
Every block $\mathcal{F}_{\lambda}$ of the category $\mathcal{F}$ of finite-dimensional $\q(3)$-modules semisimple over $\q(3)_{\bar{0}}$ is equivalent to the category of finite-dimensional modules over one of the following algebras given by quiver and relations:
\begin{enumerate}\itemsep=0pt
\item[$1.$] A strongly typical block: $\lambda = (\lambda_1,\lambda_2,\lambda_3)$ such that $\lambda_i+\lambda_j\neq 0$ and $\lambda_i\neq0$ for any $i$, $j$
\[\xymatrix{\bullet .}
\]
\item[$2.$] A typical block: $\lambda =(\lambda_1,\lambda_2,\lambda_3)$ such that some $\lambda_i=0$ and $\lambda_j+\lambda_k\neq 0$ for any $j$, $k$
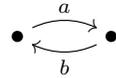
\[\xymatrix{\bullet \ar@/^/[r]^{a}& \bullet
\ar@/^/[l]^{b} }\]
with relations
 \begin{equation*}ab=ba=0.\end{equation*}
\item[$3.$] The ``half-standard'' block $\lambda = \big(\frac{3}{2},\frac{1}{2},-\frac{1}{2}\big)$
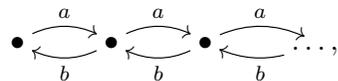
\[
\xymatrix{\bullet \ar@/^/[r]^{a}& \bullet
\ar@/^/[l]^{b}
\ar@/^/[r]^{a}& \bullet\ar@/^/[l]^{b}
\ar@/^/[r]^{a}&\dots\ar@/^/[l]^{b},}
\]
where vertices are labeled $L\big(\frac{3}{2},\frac{1}{2},-\frac{1}{2}\big)$, $L\big(\frac{5}{2},\frac{3}{2},-\frac{5}{2}\big)$, $L\big(\frac{7}{2},\frac{3}{2},-\frac{7}{2}\big)$, $\dots$
with relations \begin{equation*}a^2=b^2=0, \qquad ab=ba.\end{equation*}
\item[$4.$] The standard block $\lambda = (1,0,0)$
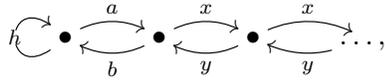
\[
\xymatrix{\bullet \ar@(dl,ul)|{h}\ar@/^/[r]^{a}& \bullet
\ar@/^/[l]^{b}
\ar@/^/[r]^{x}& \bullet\ar@/^/[l]^{y}
\ar@/^/[r]^{x}&\dots\ar@/^/[l]^{y},}
\]
where vertices are labeled $L(1,0,0)$, $L(2,1,-2)$, $L(3,1,-3)$, $\dots$
with relations
\begin{gather*}
x^2=y^2=0,\qquad xa=by=ab=0,\\
h^2=0,\qquad xy=yx,\qquad bah=hba.
\end{gather*}

\item[$5.$] The principal block $\lambda = (0,0,0)$
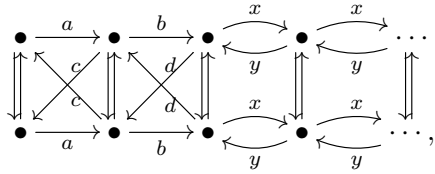
\[
\xymatrix{
\bullet \ar[r]^{a} \ar@<-0.5ex>[d] & \bullet \ar@<-0.5ex>[d]\ar[dl]^(.55){c} \ar[r]^{b} & \bullet\ar@<-0.5ex>[d]\ar@/^/[r]^x \ar[dl]^(.55){d} &\bullet \ar@<-0.5ex>[d] \ar@/^/[l]^y \ar@/^/[r]^x&\ar@<-0.5ex>[d]\cdots \ar@/^/[l]^y\\
\bullet \ar[u]\ar[r]_{a} & \bullet \ar[u]\ar[r]_{b} \ar_(.55){c}[ul] & \bullet\ar[u]\ar_(.55){d}[ul]\ar@/^/[r]^x &\bullet \ar[u]\ar@/^/[l]^y \ar@/^/[r]^x& \cdots\ar@/^/[l]^y,\ar[u]
}\]
where vertices are labeled $L(1,0,-1)$, $L(0)$, $L(2,0,-2)$, $L(3,0,-3)$, $\dots$ in top row and $\Pi L(1,0,-1)$, $\Pi L(0,0,0)$, $\Pi L(2,0,-2)$, $\Pi L(3,0,-3)$, $\dots$ in bottom row. Then labelling all vertical arrows by $\theta$, the relations are:
\begin{gather*}
 x^2=y^2=0, \qquad xb= dy=bd= ca=0,\\
 xy=yx,\qquad yx=bacd, \qquad dbac=acdb, \\
 \theta^2=0, \qquad \theta\gamma=\gamma\theta \qquad \text{for} \quad \gamma\in\{a,b,c,d, x,y\}.
\end{gather*}
\end{enumerate}
\end{thm}

\begin{cor}\label{tame wild} All blocks of $\mathfrak{sq}(3)$ are tame. The typical and standard $\q(3)$ blocks are tame. The principal $\q(3)$ block is wild.
\end{cor}
\begin{proof}
Observe that all blocks of $\sq(3)$ have special biserial quivers and hence are tame \cite{Erd}. The same holds for the two typical and standard blocks of $\q(3)$. We show the $\q(3)$ principal block is wild by ``duplicating the quiver'' \cite[Chapter~9]{GS3}. Namely, label the vertices of the quiver by $Q_0=\{1,2,3,\dots\}\cup\{-1,-2,-3,\dots\}$ corresponding to top and bottom row, respectively. Let~$Q_1$ denote the arrows and $R$ the relations. Define
$Q_0':=Q_0\cup \{1',2',3',\dots\}\cup\{-1',-2',-3',\dots\}$ and set of arrows as
\[ Q_1'=\{(i\rightarrow j')\colon (i\rightarrow j)\in Q_1\}.\]
Let $Q=(Q_0,Q_1,R)$ and $Q'=(Q_0',Q_1')$. Then $k(Q)/R'$, $R'$ being relation defined by any product of 2 arrows is~0, is a quotient of $k(Q)/R$. Note that the indecomposable representations of $(Q_0,Q_1,R')$ are in bijection with that of~$Q'$. But~$Q'$ is not a union of affine and Dynkin diagrams of type $A$, $D$, $E$ (each vertex $i, i>3$ has~3 edges coming out), so it is wild and this implies $Q$ is wild.
\end{proof}

One can also see from the description of quivers and radical filtrations of indecomposable projective modules in Appendix~\ref{appendixA} which of the blocks are highest weight categories.
\begin{cor}\label{highest weight category}
For $\mathfrak{sq}(3)$, only the blocks in cases $(1)$~$($typical$)$, $(3)$~$($half-standard$)$ and $(4)$~$($standard$)$ of Theorem~{\rm \ref{MainTheorem sq}} are highest weight categories. For~$\q(3)$, only the blocks in cases $(1)$~$($typical$)$ and $(3)$~$($half-standard$)$ are highest weight categories.
\end{cor}

\begin{proof} For different types of typical blocks the statement is obvious from the quiver.
 A half-integral block is a highest weight category both for~$\q(3)$ and~$\mathfrak{sq}(3)$. The former is also a consequence of general result in~\cite{BD} for blocks in the category of finite-dimensional
 representations of~$\q(n)$ with half-integral weights. The~$\mathfrak{sq}(3)$ standard block is also a highest weight category since it is equivalent to well
 known~$A_{\infty}$ quiver which also defines the principal block for~$\mathfrak{gl}(1|1)$~\cite{Ger}.

 The standard block for $\q(3)$ is not highest weight due to existence of self-extension.

 Let us prove now that the principal blocks for $\q(3)$ and $\mathfrak{sq}(3)$ are not highest weight categories. Note that all simple objects except $L(0)$ and $\Pi L(0)$ have zero superdimension and all projective modules have
 zero superdimension. Assume for the sake of contradiction that the principal block is a highest weight category. The isomorphism classes of simple objects $L_\mu$ are enumerated by poset~$\mathcal M$.
 Let $A_\mu$ and $P_\mu$ denote the standard and projective cover, respectively, of a simple object $L_\mu$.
 If the standard cover of $L(0)$ contain a simple constituent~$\Pi L(0)$ then the standard cover of $\Pi L(0)$ can not contain a simple constituent $L(0)$. Thus, at least one standard object has a non-zero superdimension.
 On the other hand, $P(a)$ and $\Pi P(a)$ for $a\geq 3$ do not have $L(0)$ and $\Pi L(0)$ among its simple constituents. Thus, the set of $\mu$ such that $\sdim a_\mu\neq 0$ is finite. Let us choose a maximal $\mu$ such
 $\sdim A_{\mu}\neq 0$. Then \[ \sdim P_{\mu}=\sdim A_{\mu}+\sum_{\nu>\mu}c_\nu\sdim A_{\nu}\neq 0.\]
 A contradiction. \end{proof}

\section{Geometric preliminaries and BGG reciprocity}\label{section3}

\subsection{Relative cohomology of Lie superalgebras}\label{section3.1}
Let $\mathfrak{t}\subset\g$ be a Lie subsuperalgebra and $M$ a $\g$-module. For $p\geq 0$, define
\[
C^{p}(\g, \mathfrak{t}; M)=\Hom_{\mathfrak{t}}(\wedge^{p}(\g/\mathfrak{t}), M),
\]
where $\wedge^{p}(\g)$ is the \emph{super} wedge product. The differential maps $d^{p}\colon C^{p}(\g, \mathfrak{t}; M)\to C^{p+1}(\g, \mathfrak{t}; M)$ are defined in the same way as for Lie algebras,
see for example \cite[Section~2.2]{BKN}. The {\it relative cohomology} are defined by
\[ \operatorname{H}^{p}(\g, \mathfrak{t}; M)=\Ker d^{p}/\operatorname{Im} d^{p-1}.\]

We will be interested in the case when $\mathfrak{t}=\g_{\bar{0}}$. Then the relative cohomology describe the extension groups in the category $\mathcal F$ of finite-dimensional $\g$-modules semisimple over $\g_{\bar 0}$.
More precisely, we have the following relation:
\[\Ext_{\mathcal{F}}^{p}(M,N)\cong \operatorname{H}^{p}(\g,\g_{\bar 0}; M^{*} \otimes N).\]

From here on out, we will use $\Ext_{\g}^i(-,-)$ to denote $\Ext_{\mathcal{F}}^{i}(-,-)$. For conciseness, we often write $\Ext_{\q}$ or $\Ext_{\mathfrak{sq}}$ to denote $\Ext_{\q(n)}$ or $\Ext_{\mathfrak{sq}(n)}$.

\begin{thm}\label{Ext C,C} Let $\g=\q(n)$. Then
\begin{gather*} \Ext^i_{\q(n)}(\C,\C)\cong \begin{cases} S^i(\g_{\bar{0}}^*)^{\g_{\bar0}}& \text{if $i$ even}, \\
				0 & \text{else}, \end{cases}\!\qquad \text{and}\!\!\qquad \Ext^i_{\q(n)}(\C,\Pi\C)\cong \begin{cases} S^i(\g_{\bar{0}}^*)^{\g_{\bar0}}& \text{if $i$ odd}, \\
				0 & \text{else}. \end{cases}
\end{gather*}

\end{thm}
\begin{proof} Note that $\g_1\cong\Pi\g_1$ as a $\g_{\bar 0}$-module and therefore $\Lambda^i(\g_{\bar 1}^*)\cong \Pi^i S^i(\g_{\bar 0}^*)$. Therefore
\begin{gather*}
C^i(\g,\g_{\bar 0};\mathbf C)\cong \begin{cases} S^i(\g_{\bar{0}}^*)^{\g_{\bar0}}& \text{if $i$ even}, \\
				0 & \text{else},\end{cases}\qquad \text{and}\qquad C^i(\g,\g_{\bar 0};\Pi\C)\cong \begin{cases} S^i(\g_{\bar{0}}^*)^{\g_{\bar0}}& \text{if $i$ odd}, \\
				0 & \text{else}.\end{cases}
\end{gather*}
The differential is obviously zero and the statement follows.
\end{proof}

\begin{Remark} One can also use the $\mathbf Z_2$-graded version of relative cohomology like in \cite{BKN}. It is more suitable for the superversion of the category $\mathcal F$ where odd morphisms are allowed.
\end{Remark}

\subsection{Geometric induction}\label{section3.2}
We next provide a few facts about geometric induction following the exposition in \cite{GS,PS1}. Let $\p$ be any parabolic subsuperalgebra of $\g$ containing~$\mathfrak{b}$. Let $G=Q(n)$, and $P$, $B$ be the corresponding Lie supergroups of $\p$, $\mathfrak{b}$. For a $P$ -module~$V$, we denote by the calligraphic letter $\mathcal{V}$ the vector bundle $G \times_P V$ over the generalized grassmannian $G/P$. See~\cite{Man} for the construction. Note that the space of sections of $\mathcal{V}$ on any open set has a natural structure of a $
\g$-module; in other words the sheaf of sections of $\mathcal{V}$ is a $\g$-sheaf. Therefore the cohomology groups $H^i(G/P, \mathcal{V})$ are $\g$-modules. Define the geometric induction functor $\Gamma_i$ from category of $\p$-modules to category of $\g$-modules as
\[ \Gamma_i(G/P,V):= H^i(G/P,\mathcal{V^*})^*.\]

It is also possible to define $\Gamma_i(G/P, V)$ without the need of proving the rather technical question of existence of $G/P$. Namely, consider the Zuckerman functor from the category of $P$-modules to $G$-modules
 defined by
\[ H^0(G/P,V):=\Gamma_{\mathfrak{g}_{\0}}(\Hom_{U(\mathfrak{p})}(U(\g),V)),\]
where $\Gamma_{\g_{\0}}(M)$ denotes the set of $\g_{\0}$-finite vectors of $\g$-module $M$. One can show easily that $H^0(G/P,V)$ has a unique $G$-module structure compatible with the $\g$-action. It is also straightforward that
$H^0(G/P,V)$ is left exact and the right
adjoint to the restriction functor $G\text{\rm -mod}\to P\text{\rm -mod}$. We define $H^i(G/P,\cdot)$ to be its right derived functors. Using this definition we can define $\Gamma_i(G/P,V)$ for any $V$
whose weights are in $\Lambda$.

We state some well known results.

\begin{prop}[\cite{GS, Jan}]\label{induction functor properties}
The functor $\Gamma_i$ satisfies the following properties.
\begin{enumerate}\itemsep=0pt
\item[$1.$] For any short exact sequence of $P$-modules
\[ 0\rightarrow U\rightarrow V\rightarrow W\rightarrow 0,\]
there is a long exact sequence of $\g$-modules
\[ \cdots\rightarrow\Gamma_1(G/P,W)\rightarrow\Gamma_0(G/P, U)\rightarrow\Gamma_0(G/P,V)\rightarrow\Gamma_0(G/P, W)\rightarrow 0.\]
\item[$2.$] For a $P$-module $V$ and a $\g$-module $M$,
\[\Gamma_i(G/P,V\otimes M) = \Gamma_i(G/P,V)\otimes M.\]

\item[$3.$] $\Gamma_0(G/P, V)$ is the maximal finite-dimensional quotient of $M_{\p}(V):=U(\g)\otimes_{U(\p)} V$ in the sense that any finite-dimensional quotient of $M_{\p}(V)$ is a quotient of $\Gamma_0(G/P,V)$.
\end{enumerate}
\end{prop}
If $G=Q(n)$, then all parabolic subgroups containing the standard Borel subgroup $B$ are in bijection with those of ${\rm GL}(n)$. Hence they are enumerated by partitions. The Levi subgroup~$L$ of parabolic~$P$ is isomorphic to $Q(m_1)\times \dots\times Q(m_k)$ with $m_1+\dots+m_k=n$. A weight $\lambda=(\lambda_1,\dots,\lambda_n)$ is called $\p$-typical if
\[ \lambda_i+\lambda_j=0 \qquad \text{implies}\qquad m_1+\dots+m_s<i,j\leq m_1+\dots+m_{s+1}.\]

\begin{prop}[{typical lemma, \cite[Theorem~2]{PS1}}]\label{typicallemma}
Let $P$ be any parabolic supergroup containing $B$ and suppose $\lambda\in\Lambda^+$ is $\p$-typical, where $\p:= Lie(P)$. Then
\[ \Gamma_i(G/P, L_{\p}(\lambda)) = \begin{cases}
 L(\lambda) &\text{if } i=0, \\
 0 &\text{if } i>0.
\end{cases}\]
\end{prop}

Now, for any parabolic supergroup $P$ containing $B$, define the multiplicity
\[ m_{P}^i(\lambda,\mu) := [\Gamma_i(G/P, L_{\p}(\lambda)) : L_{\g}(\mu)].\]

\begin{prop}\label{m^0 greater than Ext}
If $\lambda>\mu$, then
\[ m^0_B(\lambda,\mu)\geq\dim\Ext_{\g}^1(L(\lambda),L(\mu)).\]
\end{prop}
\begin{proof}
Suppose $0\rightarrow L(\mu)\rightarrow V\rightarrow L(\lambda)\rightarrow 0$ is an extension. Then $V$ contains a highest weight vector $v_{\lambda}$ of weight $\lambda$ coming from the inverse image of that of $L(\lambda)$. Since~$V$ is indecomposable, $V$~is generated by $v_{\lambda}$ and since $\mu<\lambda$, $V=U(\g).v_{\lambda}$ is annihilated by $\mathfrak{n}^+$. Thus~$V$ is a highest weight module of weight $\lambda$, so it is a finite-dimensional quotient of $M(\lambda)$ and consequently by Proposition~\ref{induction functor properties}(3), it is a quotient of $\Gamma_0(G/B, L_{\mathfrak{b}}(\lambda))$. Each such isomorphism class of extension $V$ thus gives rise to a distinct subquotient $L(\mu)$ in $
\Gamma_0(G/B,L_{\mathfrak{b}}(\lambda)).$ Consequently, $\dim\Ext_{\g}^1(L(\lambda),L(\mu))\leq [\Gamma_0(G/B,L_{\mathfrak{b}}(\lambda)):L(\mu)]=m^0_B(\lambda,\mu)$.
\end{proof}

\begin{Remark} In \cite{PS1}, the authors work in $\g^{\Pi}$-mod consisting of $\Pi$-invariant $\g$-modules (and even morphisms) and define $m_{P^{\Pi}}^i(\lambda,\mu)$ accordingly. For $\g=\q(n)$, the simple $\g^{\Pi}$-modules are~$L(\lambda)$ when $|\{i\colon \lambda_i\neq0\}|$ is odd and $L(\lambda)\oplus \Pi L(\lambda)$ when $|\{i\colon \lambda_i\neq0\}|$ is even.
\end{Remark}

\begin{prop}\label{m_P=m_B} Let $P$ be the parabolic subgroup of $Q(3)$ defined by roots $\{\ep_1-\ep_2,\ep_1-\ep_3,\allowbreak \ep_2-\ep_3,\ep_3-\ep_2\}$. Suppose $\lambda\in\Lambda^+\setminus\{(t,a,-a)\}$. Then for all $\mu\in\Lambda^+$, \[ m_{P}^i(\lambda,\mu)=m_B^i(\lambda,\mu).\]
\end{prop}
\begin{proof}
There is a canonical projection $G/B\rightarrow G/P$ with kernel $P/B=Q(2)/B\cap Q(2)$. By our assumption, the weights $\lambda$ is $B$-typical in $P$. Thus the Leray spectral sequence
\[ H^i(G/P, H^j(P/B, L_{\mathfrak{b}}(\lambda)))\Rightarrow H^{i+j}(G/B, L_{\mathfrak{b}}(\lambda))\]
collapses by the typical lemma.
\end{proof}

\subsection{Virtual BGG reciprocity}\label{section3.3}
We now formulate a ``virtual'' BGG reciprocity theorem for $\g=\mathfrak{sq}(n)$ or $\q(n)$ which will be used to compute composition factors of indecomposable projective covers, $P_{\g}(\lambda)$ of $L_{\g}(\lambda)$.
This result is a generalization of Theorem~1 in~\cite{GS2} in the case when Cartan subalgebra is not purely even. In this section we consider the quotient $\mathcal K^\Pi(\mathcal F)$ of the Grothendieck ring
$\mathcal K(\mathcal F)$ by the relation $[X]=[\Pi X]$. Then $\mathcal K^\Pi(\mathcal F)$ has a basis $\{[L(\lambda)]\,|\,\lambda\in\Lambda^+\}$ and $[X:L(\lambda)]_{\Pi}$ is the coefficient
$a_\lambda$ in the decomposition $[X]=\sum a_{\lambda}[L(\lambda)]$.

Denote by $\Lambda_0^+:=\{\lambda\in\h_0^*\,|\, \langle \lambda,\check{\beta}\rangle\in\mathbf{Z}_{>0},\;\forall\,\beta\in\Delta_{\overline{0}}^+\}$. For $\g=\q(n)$, $\Lambda_0^+$ consists of dominant integral weights for which at most one $\lambda_i$ is zero. For $M\in\mathcal{F}^{\Pi}$, define $\mathcal{R}:=\Z[e^{\mu}]_{\mu\in\Lambda}$ and the \text{character} of $M$
\[\text{\rm Ch}(M):=\sum_{\mu\in\Lambda}\dim(M_{\mu})e^{\mu}\in\mathcal{R},\]
where we put $\dim X:=\dim X_{\bar 0}+\dim X_{\bar 1}$.
Then $\text{\rm Ch}$ defines an injective homomorphism \mbox{$\mathcal K^\Pi(\mathcal F) \to \mathcal R$}.

For any $\lambda\in\Lambda$ we define an \textit{Euler characteristic} as
\[ \mathcal{E}(\lambda):= \sum_{\mu}\sum_{i=0}^{\dim(G/B)_{\bar{0}}}(-1)^i[\Gamma_i(G/B,v({\lambda})):L(\mu)]_{\Pi}[L(\mu)],\]
where $\Gamma_i$ is the dual to geometric induction functor as defined in Section~\ref{section2.2}.
It is straightforward to check (see, e.g., \cite[Theorem~4.25]{B1}) that for $\lambda\in\Lambda$ such that $\text{\rm wt}(\lambda)=\gamma$, then $[\mathcal{E}(\lambda)]\in\mathcal{K}^{\Pi}(\mathcal{F}_{\gamma})$.

Let us comment on the relation between this Euler characteristic the one defined in~\cite{B2}. There, the author considered an induction from the maximal parabolic $P_{\lambda}$ to which $v({\lambda})$ extends, i.e.,
\[\mathcal{E}_{P}(\lambda):= \sum_{\mu}\sum_{i=0}^{\dim(G/P_{\lambda})_{\bar{0}}}(-1)^i[\Gamma_i(G/P_{\lambda},v({\lambda})):L(\mu)][L(\mu)].\]
If $\lambda\in\Lambda^+$ is regular then $P=B$ and $\mathcal{E}_{P}(\lambda)=\mathcal{E}(\lambda)$ and if $\lambda$ is not regular $\mathcal{E}(\lambda)=0$ while $\mathcal{E}_{P}(\lambda)\neq 0$. It was shown in~\cite{B2}
that $\mathcal{E}_{P}(\lambda)$ form a basis of the Grothendieck group of~$\mathcal F$.

 The following result is a straightforward generalization of \cite[Lemma 1.2]{GS}.
\begin{Lemma}\label{Euler Characteristic} The Euler characteristic $\mathcal{E}(\lambda)$ satisfies
\begin{enumerate}\itemsep=0pt
\item[$1.$]
\begin{equation*}
\text{\rm Ch}(\mathcal{E}(\lambda))=\dim v({\lambda})D\sum_{w\in S_n}\ep(w)e^{w.\lambda},
\end{equation*}
where
\[ D=\prod_{\alpha\in\Phi^+}\frac{e^{\alpha/2}+e^{-\alpha/2}}{e^{\alpha/2}-e^{-\alpha/2}}.\]
\item[$2.$] For all $w\in W$, \[ \mathcal{E}(\lambda) =\ep(w)\mathcal{E}(w.\lambda).\]
\item[$3.$] Let $\Lambda^+_0$ denote the set of regular dominant weights with respect to $\g_{\bar 0}$. The set \[\{\text{\rm Ch}(\mathcal{E}(\lambda)), \lambda\in\Lambda_0^+\}\] is linearly independent in the ring~$\mathcal{R}$.
\end{enumerate}
\end{Lemma}
We call a simple $\g$-module $L(\lambda)$ of \textit{type M} if $\Pi L(\lambda)$ is not isomorphic to $L(\lambda)$ and of \textit{type Q} if $\Pi L(\lambda)\cong L(\lambda)$.
Note that the type of $L(\lambda)$ is the same as the type of~$v(\lambda)$. Furthermore, for $\g=\q(n)$ the type depends on the number of non-zero entries in $\lambda$: the type is~$M$, if this number is even, and~$Q$ if it is odd.
For example, $L(1,0,0)$ is of type Q and $L(0)$ is of type M. We set
\[ t(\nu)=\begin{cases}1& \text{if } L(\nu) \ \text{type } M,\\ 0 & \text{if } L(\nu) \ \text{type } Q.\end{cases}
\]

\begin{thm}\label{BGG Reciprocity} Let $\g=\q(n)$ or $\mathfrak{sq}(n)$. Let $\mu\in\Lambda^+$ and $b_{\mu,\lambda}$ be the coefficients occurring in the expansion
\[\mathcal{E}(\mu) = \sum_{\lambda\in\Lambda^+} b_{\mu,\lambda}[L(\lambda)].\]
Then there exists coefficients $a_{\lambda,\mu}$ such that for $\lambda\in\Lambda^+$,
\[ [P(\lambda)] = \sum_{\mu\in\Lambda_0^+}a_{\lambda,\mu}\mathcal{E}(\mu)\]
and
\[ a_{\lambda,\mu} = 2^{t(\mu)-t(\lambda)}\gamma_{\mu}b_{\mu,\lambda},\]
where
\[
\gamma_{\mu}=\begin{cases}1& \text{if}\ \g=\q(n)\ \text {and } \prod\mu_i\neq 0,\ \text{or}\ \g=\mathfrak{sq}(n) \ \text{and}\ \sum\frac{1}{\mu_i}\neq 0, \\ 2& \text{otherwise}.\end{cases}
\]
\end{thm}

\begin{proof}
We follow the proof of \cite[Theorem~1]{GS2}. First, we have the Bott reciprocity formula
\begin{equation}\label{Bott}
 \dim\Hom_{\g}(P(\lambda),\Gamma_i(V))=\dim\Ext^i_B(V,P(\lambda))=\dim H^i(\mathfrak b,\h_{\bar 0};V^*\otimes P(\lambda)).
 \end{equation}
Let $C^i(\mathfrak n,-)$ stand for the $i$-th term of the cochain complex computing $H^{\bullet}(\mathfrak n,-)$. Note that $P(\lambda)$ and hence $C^i(\mathfrak n;V^*\otimes P(\lambda))$ is projective and injective in the category of
$\h$-modules semisimple over $\h_{\bar 0}$.
Hence $H^j(\h,\h_{\bar 0};C^i(\mathfrak n,V^*\otimes P(\lambda)))=0$ for any $i$ and $j\geq 1$. Therefore the first term of the spectral sequence for the pair
$(\mathfrak b,\h)$
implies that
\begin{equation}\label{spectral}
 \sum_{i=0}^\infty(-1)^i\dim\Ext^i_B(v(\mu),P(\lambda))=\sum_{i=0}^\infty(-1)^i\dim\Hom_{\h}(v(\mu),C^i(\mathfrak n,P(\lambda))).
 \end{equation}
 Furthermore, we have
 \begin{equation}\label{multiplicity}
 [M:L(\lambda)]_{\Pi}=\begin{cases} \dim\Hom_{\g}(P(\lambda)\oplus\Pi P(\lambda), M) &\text{if } L(\lambda)\text{ type } M,\\ \dim\Hom_{\g}(P(\lambda), M)&\text{if } L(\lambda) \text{ type } Q.	\end{cases}
 \end{equation}
Define $b^i_{\mu,\lambda}$ by
\[ b^i_{\mu,\lambda}:=\begin{cases}\dim\Hom_{\h}(v(\mu)\oplus\Pi v(\mu), C^i(\mathfrak n,P(\lambda))) &\text{if } L(\lambda) \text{ type } M,\\ \dim\Hom_{\h}(v(\mu), C^i(\mathfrak n,P(\lambda)))&\text{if } L(\lambda)\text{ type } Q. \end{cases}
\]
By application of (\ref{spectral}) and (\ref{multiplicity}) we obtain
\[ b_{\mu,\lambda}=\sum_{i=0}^\infty(-1)^ib^i_{\mu,\lambda}.\]
For any module $M\in\mathcal F$ projective over $\h$ we have the equality
\begin{equation}\label{multchar}
 \frac{\dim M_\mu}{\dim\hat{v}(\mu)}=\begin{cases} \dim\Hom_{\h}(v(\mu)\oplus\Pi v(\mu),M)&\text{if } v(\mu)\text{ type } M,\\ \dim\Hom_{\h}(v(\mu),M)&\text{if } v(\mu)\text{ type } Q, \end{cases}
 \end{equation}
where $\hat{v}(\mu)$ is the corresponding indecomposable injective $\h$-module.
In other words we get
\[\text{\rm Ch}(M)=\sum_{\mu\ \text{type M}}\dim\Hom_{\h}(v(\mu)\oplus\Pi v(\mu),M)e^\mu+\sum_{\mu \ \text{type Q}}\dim\Hom_{\h}(v(\mu),M)e^\mu.\]
If $\lambda$ is of type Q we obtain
\begin{align*}
\text{\rm Ch}(C^i(\mathfrak n,P(\lambda)))& =\sum_{\mu\ \text{type M}}2b_{\mu,\lambda}^i\dim \hat{v}(\mu)e^\mu+\sum_{\mu \ \text{type Q}}b_{\mu,\lambda}^i\dim \hat{v}(\mu)e^\mu\\
& =\sum_{\mu}2^{t(\mu)-t(\lambda)}b_{\mu,\lambda}^i\dim \hat{v}(\mu)e^\mu.
\end{align*}
If $\lambda$ is of type M we obtain
\begin{align*}
\text{\rm Ch}(C^i(\mathfrak n,P(\lambda)))& =\sum_{\mu\ \text{type M}}b_{\mu,\lambda}^i\dim \hat{v}(\mu)e^\mu+\sum_{\mu \ \text{type Q}}\frac{1}{2}b_{\mu,\lambda}^i\dim \hat{v}(\mu)e^\mu\\
& =\sum_{\mu}2^{t(\mu)-t(\lambda)}b_{\mu,\lambda}^i\dim \hat{v}(\mu)e^\mu.
\end{align*}
Taking alternating sum over $i$ we get
\[ \sum_{i=1}^{\infty}(-1)^i\text{\rm Ch}(C^i(\mathfrak n,P(\lambda)))=\sum_{\mu}2^{t(\mu)-t(\lambda)}b_{\mu,\lambda}\dim \hat{v}(\mu)e^\mu.\]
On the other hand, we have
\[ \sum_{i=1}^{\infty}(-1)^i\text{\rm Ch}(C^i(\mathfrak n,P(\lambda)))=\text{\rm Ch}(P(\lambda))\prod_{\alpha\in\Phi^+}\frac{1-e^{-\alpha}}{1+e^{-\alpha}}=D^{-1}\text{\rm Ch}(P(\lambda)).\]
This implies
\[ \text{\rm Ch}(P(\lambda))=D\sum_{\mu\in\Lambda}b_{\mu,\lambda}\dim\hat{v}(\mu)2^{t(\mu)-t(\lambda)}e^\mu.\]
By $S_n$-invariance of $\text{\rm Ch}(P(\lambda))$, we get
\[ b_{\mu,\lambda}=\varepsilon(w)b_{w.\mu,\lambda}\qquad\forall\, w\in S_n.\]
This together with $\dim\hat{v}(\mu)=\dim\hat{v}(w.\mu)$ implies
\begin{align*}
\text{\rm Ch}(P(\lambda))& =D\sum_{w\in W}\sum_{\mu\in\Lambda^+_0}b_{\mu,\lambda}\varepsilon(w)\dim\hat{v}(\mu)2^{t(\mu)-t(\lambda)}e^{w.\mu}\\
& =\sum_{\mu\in\Lambda^+_0}\frac{\dim \hat{v}(\mu)}{\dim v(\mu)}2^{t(\mu)-t(\lambda)}b_{\mu,\lambda}\text{\rm Ch}(\mathcal E(\mu)).
\end{align*}
Therefore we obtain the relation
\begin{equation}\label{fundamental}
 a_{\lambda,\mu}=\frac{\dim \hat{v}(\mu)}{\dim v(\mu)}2^{t(\mu)-t(\lambda)}b_{\mu,\lambda}.
\end{equation}

Since $\mu\in\Lambda_0^+$ at most one $\mu_i=0$. Therefore, we get: for $\g=\q(n)$, $v(\mu)=\hat{v}(\mu)$ if all $\mu_i\neq 0$; for $\g=\mathfrak{sq}(n)$, $v(\mu)=\hat{v}(\mu)$ if $\sum_{i=1}^n\frac{1}{\mu_i}\neq 0$.
In remaining cases $\frac{\dim \hat{v}(\mu)}{\dim v(\mu)}=2$.
\end{proof}

\begin{Remark} Theorem \ref{BGG Reciprocity} holds for any Lie superalgebra $\g$ such that $\h=\h_{\bar0}$ and $\g_{\bar1}=\g_{\bar1}^*$. In this case, we get $\gamma_{\mu}=1$.
\end{Remark}

Let $\mathcal{K}^{\Pi}_P(\mathcal{F})$ be the subgroup of $\mathcal{K}^\Pi(\mathcal{F})$ generated by the classes of all projective modules. It is an ideal in $\mathcal{K}^{\Pi}(\mathcal{F})$ since tensor product of projective with any finite-dimensional module is projective. Let $\mathcal{K}^\Pi_E(\mathcal{F})$ be the subgroup of $\mathcal{K}^{\Pi}(\mathcal{F})$ generated by the Euler characteristics. Then $\mathcal{K}^{\Pi}_P(\mathcal{F})\subset\mathcal{K}^{\Pi}_E(\mathcal{F})\subset\mathcal{K}^{\Pi}(\mathcal{F})$ and the inclusions are in general strict. The $b_{\nu,\mu}$ express the basis of $\mathcal{K}^{\Pi}_E(\mathcal{F})$ in terms of
the basis of $\mathcal{K}^{\Pi}(\mathcal{F})$ and $a_{\lambda,\nu}$ express the basis of $\mathcal{K}^{\Pi}_P(\mathcal{F})$ in terms of the basis of $\mathcal{K}^{\Pi}_E(\mathcal{F})$. Thus for two $\g$-dominant weights $\lambda,\mu$, we have
\begin{equation}\label{Projective multiplicity}
[P(\lambda):L(\mu)]_\Pi=\sum_{\nu\in\Lambda_0^+}a_{\lambda,\nu}b_{\nu,\mu}.
\end{equation}
\begin{Remark} In \cite{B1} the coefficients $b_{\mu,\lambda}$ and the multiplicities $[P(\lambda):L(\mu)]$ were computed using the action of the Kac--Moody superalgebra $B_{\infty}$ on $\mathcal F$ via translation functors.
 Since $[P(\lambda)]$ and $[L(\lambda)]$ form a dual system in $\mathcal{K}^{\Pi}(\mathcal{F})$ ($\dim\Hom_{\mathfrak q}(P(\lambda),L(\mu))=\delta_{\lambda,\mu}$) the action of translation functors on $[P(\lambda)]$ is related to the action on $[L(\lambda)]$ in the natural way via this duality.
 Applying translation functors repeatedly starting from a typical representation, the author obtains a nice combinatorial formula
 for $b_{\mu,\lambda}$. In addition, it gives another way to prove Theorem~\ref{BGG Reciprocity} in this particular case.
\end{Remark}

\subsection{General lemma}\label{generalind} To study relations between block for $\mathfrak{sq}(n)$ and $\q(n)$ we consider the induction and restriction functors
\begin{gather*}
\Ind\colon \ \mathcal F_{\mathfrak{sq}(n)}\to \mathcal F_{\mathfrak{q}(n)},\qquad M\mapsto \Ind^{\q(n)}_{\mathfrak{sq}(n)}M,\\
\Res\colon \ \mathcal F_{\mathfrak{q}(n)}\to \mathcal F_{\mathfrak{sq}(n)},\qquad M\mapsto \Res_{\mathfrak{sq}(n)}M.
\end{gather*}
The Frobenius reciprocity implies that $\Ind$ is left adjoint of $\Res$.

\begin{Lemma}\label{general} Let $M$ be a projective $\mathfrak{sq}(n)$-module with $\Pi M\cong M$ and let $\mathcal A=\End_{\mathfrak{sq}}(M)$, $\mathcal A'=\End_{\mathfrak{q}}(\Ind M)$. Assume that there exists $\theta\in\mathcal A'$
 such that $\Ker\theta=\operatorname{Im}\theta$ and $\Ker\theta\cap (1\otimes M)=\{0\}$. Then $\mathcal A'\cong \mathcal A\otimes\C[\theta]/\big(\theta^2\big)$.
 \end{Lemma}
 \begin{proof} Note that our assumptions imply $\Res \Ind M=M\oplus M$.
 Consider injective homomorphism $\Ind\colon \mathcal A\to\mathcal A'$ and $\Res\colon \mathcal A'\to \text{Mat}_2\otimes A$.
 Furthermore, for $\gamma\in \mathcal A$, we have
\[ \Res\Ind\gamma=\left(\begin{matrix}\gamma&\gamma'\\ 0&\gamma\end{matrix}\right)\]
 for some $\gamma'\in\mathcal A$.
 The condition $\theta^2=0$ implies
\[ \Res \theta=\left(\begin{matrix}0& 0\\ \text{\rm Id}&0\end{matrix}\right).\]

 The Frobenius reciprocity implies for any $\varphi\in\mathcal A'$,
 \[ \text{if}\quad\Res\varphi=\left(\begin{matrix}0& \sigma\\ 0&\tau\end{matrix}\right)\qquad\text{then}\qquad \Res\varphi=0.
 \]
 We have
\begin{gather*}
 [\Res\Ind\gamma,\Res\theta]=\left(\begin{matrix}\gamma'& 0\\ 0&-\gamma'\end{matrix}\right),\qquad
 [\Res\Ind\gamma,\Res\theta]-\Res\Ind\gamma'=\left(\begin{matrix}0& -\gamma''\\ 0&-2\gamma'\end{matrix}\right),
\end{gather*}
 hence $\gamma'=0$.

 Thus, we have proved that $\Ind(\mathcal A)$ commutes with~$\theta$. Thus there is an injective homomorphism $\mathcal A\otimes \C[\theta]/\big(\theta^2\big)\to \mathcal A$. The dimension argument implies that it is an isomorphism.
 \end{proof}

\section{Self extensions}\label{section4}

\subsection[Self extensions for q(n)]{Self extensions for $\boldsymbol{\q(n)}$}\label{section4.1}
The goal of this section is to prove the following theorem.
\begin{thm}\label{selfextensions}
Let $\lambda = (\lambda_1,\dots,\lambda_k,0,\dots,0, -\lambda_{k+m+1},\dots,-\lambda_n)\in\Lambda^+$ such that all $\lambda_i>0$ be a~dominant integral weight in $\q(n)$. Then
\[\Ext_{\q(n)}^1(L(\lambda),\Pi L(\lambda))=\begin{cases} \C&\text{if } m>0,\\
		0&\text{if } m=0.\end{cases}
\]
If $L(\lambda)\neq\Pi L(\lambda)$, then
\[\Ext_{\q(n)}^1(L(\lambda), L(\lambda))=0.\]
\end{thm}

Theorem \ref{selfextensions} implies parts (1) and (2) of our main theorem \ref{MainTheorem}. Namely, $\Ext^1(L(\lambda),L(\mu))\neq 0\Rightarrow \text{\rm wt}(\lambda)=\text{\rm wt}(\mu)$. Thus, by Theorem \ref{selfextensions}, there are no extensions in the strongly typical block, and there is a unique extension $\frac{L(\lambda)}{\Pi L(\lambda)}$ in the typical block. Thus in the typical block, the projective cover of $L(\lambda)$ is $P(\lambda) = \frac{L(\lambda)}{\Pi L(\lambda)}$ (Theorem \ref{BGG Reciprocity}). Then $a\in\Hom_{\q} (P(\lambda),\Pi P(\lambda))$ implies $a^2=0$.

\begin{proof}
The key idea is to take parabolic invariants to reduce the problem to finding extensions between trivial modules.
Let $\lambda$ be as in the theorem. Define the parabolic subalgebra of $\g$ by
\[\mathfrak{p}:=\h\oplus\bigoplus_{1<i<j\leq n}\g_{\ep_i-\ep_j}\oplus\bigoplus_{k<i<j\leq k+m}\g_{\ep_j-\ep_i}.\]
Its Levi subalgebra $\mathfrak{l}\subset\p$ is isomorphic to $\q(m)\oplus \h'$ where $\h'\subset\h$ is the centralizer of $\q(m)$ in $\h$. Let
\[n_{\p}:=\bigoplus_{i\leq k<j\leq n}\g_{\ep_i-\ep_j}\oplus\bigoplus_{k<i\leq k+m<j\leq n}\g_{\ep_i-\ep_j}\]
be the nilpotent radical of $\p$.

We first observe that taking $n_{\p}$ invariants is a functor from $\q(n)$-mod to $\mathfrak{l}$-mod. Next, suppose $L(\lambda)^{n_{\p}}$ had a nontrivial $\mathfrak{l}$-invariant subspace $N$. Because $\mathfrak{l}$ preserves the $\lambda$-weight space, and the lower parabolic nilpotent part only lowers the $\lambda$-weight space, we must have $U(\q(n))N_{\lambda}\subsetneq L(\lambda)_{\lambda} \Rightarrow U(\q(n))N \subsetneq L(\lambda)\Rightarrow N=0$, contradiction. Thus $L(\lambda)^{n_{\p}}$ is simple $\mathfrak{l}$-module.
On the other hand, $L(\lambda)_{\lambda}$ is also an irreducible $\mathfrak{l}$-module of highest weight $\lambda$. So by the characterization of the simple highest weight $\mathfrak{l}$-modules, $L(\lambda)^{n_{\mathfrak{p}}}=L(\lambda)_{\lambda}$.

\begin{Lemma}\label{lemma1self} Using the above notation, the following linear maps are injective
\begin{gather*} \Ext_{\q(n)}^1(L(\lambda),L(\lambda))\hookrightarrow \Ext_{\mathfrak{l}}^1(L(\lambda)^{n_{\mathfrak{p}}},L(\lambda)^{n_{\mathfrak{p}}}),\\
\Ext_{\q(n)}^1(L(\lambda),\Pi L(\lambda))\hookrightarrow \Ext_{\mathfrak{l}}^1(L(\lambda)^{n_{\mathfrak{p}}},\Pi L(\lambda)^{n_{\mathfrak{p}}}).
\end{gather*}
\end{Lemma}
\begin{proof}
 Suppose we had a sequence of $\q(n)$-modules $0\rightarrow L(\lambda)\rightarrow M\rightarrow L(\lambda)\rightarrow 0$ such that taking $n_{\p}$ invariants results in a split short exact sequence of $\mathfrak{l}$-modules
\[ 0\rightarrow L(\lambda)^{n_{\mathfrak{p}}}\xrightarrow{\phi} M^{n_{\p}}\xrightarrow{\psi} L(\lambda)^{n_{\mathfrak{p}}}
\rightarrow 0.\]
 From before, we know this sequence is the same as
\[0\rightarrow L(\lambda)_{\lambda}\xrightarrow{\phi} M_{\lambda}\xrightarrow{\psi} L(\lambda)_{\lambda}
\rightarrow 0.\]

Existence of a splitting maps means there exists an $\mathfrak{l}$-module homomorphism $\delta\colon L(\lambda)_{\lambda}\rightarrow M_{\lambda}$ such that $\psi\circ \delta =\text{id}_{L(\lambda)_{\lambda}}$. Thus, we know
\[ M_{\lambda} = \phi(L(\lambda)_{\lambda})\oplus \delta(L(\lambda)_{\lambda}).\]
Let $L' = U(\q(n)).\delta(L(\lambda)_{\lambda})$. Then
\begin{align*}
M=U(\q(n)).M_{\lambda} &= U(\q(n)).(\phi(L(\lambda)_{\lambda})\oplus \delta(L(\lambda)_{\lambda}))\\
&\subset U(\q(n))\phi(L(\lambda)_{\lambda})+ U(\q(n)) \delta(L(\lambda)_{\lambda}) \\
&  =\phi(L(\lambda))+L'
\end{align*}
where in the last line we use that $\phi$ is a $\q(n)$-module homomorphism and $U(\q(n))L(\lambda)_{\lambda}=L(\lambda)$. Thus $M=\phi(L(\lambda))+ L'$.

To show the sum is direct, observe $\phi$ injective and $L(\lambda)$ simple implies $\phi(L(\lambda))$ is simple, so $\phi(L(\lambda))\cap L'= 0$ or $\phi(L(\lambda))$. The case
$\phi(L(\lambda))\cap L'= \phi(L(\lambda))$ is impossible, as
\[(L'\cap\Phi(L(\lambda)))_\lambda=L'_\lambda\cap(\Phi(L(\lambda)))_\lambda =0.\]
This shows that the sequence of $\q(n)$ modules splits also.
\end{proof}

If $m=0$, then $\Ext^1_{\q(n)}(L(\lambda),L(\lambda))=0$. This follows from
\[\dim\Ext_{\mathfrak{q}(n)}^1(L(\lambda),\Pi L(\lambda))\leq \dim\Ext_{\mathfrak{h}(n)}^1(L(\lambda)^{n_{\p}},\Pi L(\lambda)^{n_{\p}})=0,\]
 where we used $m=0\Rightarrow \mathfrak{l} = \mathfrak{h}$ and Lemma~\ref{lemma1self} for the first step, and the well known fact that Clifford supermodules are semisimple when $\lambda$ is nondegenerate, for the second step \cite{Mein}.
\begin{Lemma} \label{lemma2self} If $\lambda$ is as in Theorem~{\rm \ref{selfextensions}} and $m>0$ and $v(\lambda)$ is considered as a simple $\mathfrak l$-module, then $\Ext^1_{\mathfrak l}(v(\lambda),\Pi v(\lambda))= \C$
 and $\Ext^1_{\mathfrak l}(v(\lambda),v(\lambda))=0$ if $v(\lambda)$ is of Type M.
 \end{Lemma}

\begin{proof}We start with general observation.
\begin{Lemma}\label{eqcohomology} Suppose $\g=A\oplus B$, where $A,B$ are Lie superalgebras and $M=M_A\boxtimes M_B$ is a~$\g$-supermodule. Then
\begin{gather*}
H^1(\g,\g_0; M) =H^1(A,A_0; M_A)\boxtimes H^0(B,B_0; M_B)\oplus H^1(A,A_0; \Pi M_A)\boxtimes H^0(B,B_0; \Pi M_B)\\
\hphantom{H^1(\g,\g_0; M) =}{} \oplus H^0(A,A_0; M_A)\boxtimes H^1(B,B_0; M_B)\\
\hphantom{H^1(\g,\g_0; M) =}{} \oplus H^1(A,A_0; \Pi M_A)\boxtimes H^1(B,B_0; \Pi M_B).
\end{gather*}
\end{Lemma}

Now write $v(\lambda)=\mathbf C\boxtimes v(\lambda')$ for $\mathfrak l=\q(m)\oplus\h'$.
Then since $v(\lambda')$ is a projective $\h'$-module we have
\begin{gather*}
\Ext^1_{\mathfrak l}(v(\lambda),\Pi v(\lambda))=\Ext^1_{\q(m)}(\mathbf C,\mathbf C)\otimes \Hom_{\h'}(v(\lambda'),\Pi v(\lambda'))\\
\hphantom{\Ext^1_{\mathfrak l}(v(\lambda),\Pi v(\lambda))=}{}
\oplus\Ext^1_{\q(m)}(\mathbf C,\Pi\mathbf C)\otimes \Hom_{\h'}(v(\lambda'),v(\lambda'))
\end{gather*}
and
\begin{gather*}
\Ext^1_{\mathfrak l}(v(\lambda),v(\lambda))=\Ext^1_{\q(m)}(\mathbf C,\mathbf C)\otimes \Hom_{\h'}(v(\lambda'),v(\lambda'))\\
\hphantom{\Ext^1_{\mathfrak l}(v(\lambda),v(\lambda))=}{}
\oplus\Ext^1_{\q(m)}(\mathbf C,\Pi\mathbf C)\otimes \Hom_{\h'}(v(\lambda'),\Pi v(\lambda')).
\end{gather*}
 By Theorem~\ref{Ext C,C} we have $\Ext^1_{\q(m)}(\mathbf C,\mathbf C)=0$ and $\Ext^1_{\q(m)}(\mathbf C,\Pi\mathbf C)=\mathbf C$. The lemma follows.
\end{proof}

By Lemmas~\ref{lemma1self} and~\ref{lemma2self} we have that
\[\dim \Ext^1_{\q(n)}(L(\lambda),\Pi L(\lambda))\leq 1\qquad \text{and}\qquad \Ext^1_{\q(n)}(L(\lambda),L(\lambda))=0\]
if $L(\lambda)$ is not isomorphic to $\Pi L(\lambda)$.

It remains to show that there exists a non-trivial extension between $L(\lambda)$ and $\Pi L(\lambda)$. For this we consider an indecomposable $(1|1)$-dimensional $q(n)$-module $U$ with a basis
$u\in U_{\bar 0}$, $\bar u\in U_{\bar 1}$ and with action given by
\[ X\bar u=0,\qquad X u=\operatorname{otr}X \bar u, \qquad \forall\, \ X\in\q(n).\]
Then we have an exact sequence of $\q(n)$-modules
\[ 0\to \Pi L(\lambda)\to L(\lambda)\otimes U\to L(\lambda)\to 0.\]
To see that it does not split take $p$ such that $\lambda_p=0$. On the weight space $(L(\lambda)\otimes U)_\lambda$ the odd basis element $\overline{H}_p$ acts non-trivially while
its action on $(L(\lambda)\oplus\Pi L(\lambda))_\lambda$ is obviously trivial.
The proof of Theorem \ref{selfextensions} is complete.
\end{proof}

\subsection[Self extensions for sq(n)]{Self extensions for $\boldsymbol{\mathfrak{sq}(n)}$}\label{section4.2}
\begin{thm}\label{selfextensions2}
 Let $\lambda = (\lambda_1,\dots,\lambda_k,0,\dots,0, -\lambda_{k+m+1},\dots,-\lambda_n)$ such that all $\lambda_i>0$ be a dominant integral weight in $\q(n)$.
 \begin{enumerate}\itemsep=0pt
\item[$1.$] $\Ext_{\sq(n)}^1(L_{\sq(n)}(\lambda), L_{\sq(n)}(\lambda))=0$ if $L(\lambda)\neq\Pi L(\lambda)$;
\item[$2.$] $\Ext_{\sq(n)}^1(L_{\sq(n)}(\lambda),\Pi L_{\sq(n)}(\lambda))=\Ext_{\sq(n)}^1(L_{\sq(n)}(\lambda), L_{\sq(n)}(\lambda))=0$
if $m>0$ or $m=0$ and $\frac{1}{\lambda_1}+\dots+\frac{1}{\lambda_n}\neq 0$;
\item[$3.$] $\Ext_{\sq(n)}^1(L_{\sq(n)}(\lambda),\Pi L_{\sq(n)}(\lambda))=\C$ if $\frac{1}{\lambda_1}+\dots+\frac{1}{\lambda_n}=0$.
\end{enumerate}
\end{thm}
\begin{proof} Note that Lemma~\ref{lemma1self} can be generalized to the case of $\sq(n)$, namely if $\mathfrak{p}'=\p\cap\sq(n)$, $\mathfrak{l}'=\mathfrak{l}\cap\sq(n)$ and $v'(\lambda)$ is the irreducible $\mathfrak{l}'$-module, the map
\[ \Ext_{\sq(n)}^1(L_{\sq(n)}(\lambda),(\Pi)L_{\sq(n)}(\lambda))\hookrightarrow \Ext_{\mathfrak{l}'}^1(v'(\lambda),(\Pi)v'(\lambda))\]
is injective.
We claim that $\Ext_{\mathfrak{l}'}^1(v'(\lambda),(\Pi)v'(\lambda))=0$ for all $\lambda$ which do not satisfy (3). Indeed, if $m=0$ and $\frac{1}{\lambda_1}+\dots+\frac{1}{\lambda_n}\neq 0$, $K'_\lambda=0$ (see Section~\ref{section2.3}) and hence
$v'(\lambda)$ is projective. If $m>0$, then $\sq(m)$ is an ideal in $\mathfrak{l}'$ which acts trivially on $v'(\lambda)$ and $\Pi v'(\lambda)$ and the quotients $\mathfrak{l}'/\sq(m)\cong \mathfrak{l}/\q(m)\cong \h'$. Since
\[\Ext^1_{\sq(m)}(\C,\C)=\Ext^1_{\sq(m)}(\C,\Pi\C)=0,\]
using spectral sequence we get
\[H^1(\mathfrak{l}',\mathfrak{l}'_{\bar{0}}, v'(\lambda)^*\otimes (\Pi)v'(\lambda))\cong H^1(\h',\h'_{\bar{0}}, v'(\lambda)^*\otimes (\Pi)v'(\lambda))=0.\]

It remains to consider the case $\frac{1}{\lambda_1}+\dots+\frac{1}{\lambda_n}=0$. In this case one-dimensional $K'_\lambda$ lies in the radical of the corresponding Clifford algebra, therefore we have
\[\Ext^1_{\mathfrak{l}'}(v'(\lambda),\Pi v'(\lambda))=\C.\]
To construct a non-trivial extension over $\sq(n)$ consider the induced module $\mathrm{Ind}^{\q(n)}_{\sq(n)}L_{\sq(n)}(\lambda)$ isomorphic to $L(\lambda)=L_{\q(n)}(\lambda)$. It is the middle term of an exact sequence of $\sq(n)$-modules
\[0\to L_{\sq(n)}(\lambda)\xrightarrow{\varphi} L(\lambda)\to \Pi L_{\sq(n)}(\lambda)\to 0.\]
Let us check that the sequence does not split. Using Frobenius reciprocity, we compute
\[\Hom_{\sq(n)}(\Pi L_{\sq(n)}(\lambda), L(\lambda))=\operatorname{Hom}_{\q(n)}(\Pi L(\lambda),L(\lambda)).\]
If $n$ is odd, then $L_{\sq(n)}(\lambda)$ is isomorphic to $\Pi L_{\sq(n)}(\lambda)$ and $L(\lambda)$ is isomorphic to $\Pi L(\lambda)$ as a~$\q(n)$-module. So
\[\C\varphi=\Hom_{\sq(n)}(\Pi L_{\sq(n)}(\lambda), L(\lambda))=\operatorname{Hom}_{\q(n)}(\Pi L(\lambda),L(\lambda)).\]
If $n$ is even then $L_{\sq(n)}(\lambda)$ is not isomorphic to $\Pi L_{\sq(n)}(\lambda)$ and $L(\lambda)$ is not isomorphic to $\Pi L(\lambda)$ as an $\q(n)$-module.
Therefore
\[ \Hom_{\sq(n)}(\Pi L_{\sq(n)}(\lambda), L(\lambda))=\operatorname{Hom}_{\q(n)}(\Pi L(\lambda),L(\lambda))=0.\]
In both cases, the sequence does not split.
 \end{proof}

 \begin{cor}\label{Ind-Res}
 Let $\lambda\in\Lambda^+$ and let $\Res$, $\Ind$ denote $\Res_{\sq}^{\q}$, $\Ind_{\sq}^{\q}$ respectively.
 \begin{enumerate}\itemsep=0pt
 \item[$(a)$] If there exists $i$ such that $\lambda_i=0$, then
\[ \Res L(\lambda) = L_{\sq}(\lambda),\qquad \Ind L_{\sq}(\lambda) =\frac{L(\lambda)}{\Pi L(\lambda)}.\]
 \item[$(b)$] If all $\lambda_i\neq 0$ and $\frac{1}{\lambda_1}+\dots+\frac{1}{\lambda_n}\neq 0$, then
\[\Res L(\lambda) = L_{\sq}(\lambda)\oplus \Pi L_{\sq}(\lambda),\qquad \Ind L_{\sq}(\lambda) = L(\lambda).\]
 \item[$(c)$] If all $\lambda_i\neq 0$ and $\frac{1}{\lambda_1}+\dots+\frac{1}{\lambda_n}= 0$, then
\[\Res L(\lambda) = \frac{\Pi L_{\sq}(\lambda)}{ L_{\sq}(\lambda)}, \qquad \Ind L_{\sq}(\lambda) = L(\lambda).\]
\end{enumerate}
 \end{cor}

 \begin{proof} By PBW theorem for $\q(n)$, $\sq(n)$, given a finite-dimensional $\sq(n)$ module $M$, \linebreak $\dim\Res\Ind M = 2\dim M$. Suppose we are in case~(a). By Lemma~\ref{dim Clifford}(a), $\Res L(\lambda) = L_{\sq}(\lambda)$ and $\Ind L_{\sq}(\lambda)$ is the middle term of the exact sequence of $\q(n)$ modules
\[ 0\rightarrow L(\lambda)\xrightarrow{\phi} \Ind L_{\sq}(\lambda)\rightarrow \Pi L(\lambda)\rightarrow 0.\]
Now, since there some $\lambda_i=0$, $L_{\sq}(\lambda) \cong\Pi L_{\sq}(\lambda)$ if and only if $L(\lambda)\cong\Pi L(\lambda)$. Then repeating the argument from Theorem~\ref{selfextensions2}, we conclude the sequence is nonsplit. For case~(b), note that Lemma~\ref{dim Clifford}(b) implies $L(\lambda) \cong \Pi L(\lambda)$ if and only if $L_{\sq}(\lambda)\not\cong\Pi L_{\sq}(\lambda)$, and Theorem~\ref{selfextensions2} implies there is no self extension $\frac{L_{\sq}(\lambda)}{\Pi L_{\sq}(\lambda)}$. Finally, case~(c) was done in Theorem~\ref{selfextensions2}.
 \end{proof}

\section{Standard block}\label{section5}
In this section, we compute the Ext-quiver for the standard block of $\g=\q(3)$ and $\g=\mathfrak{sq}(3)$.

\subsection{Induction and restriction functors}\label{section5.1}
 \looseness=-1 Our goal is to establish a connection between the standard block $\mathcal F^n_{(1,0,\dots,0)}$ and the principal block $\mathcal F^{n-1}_{(0)}$. As a first step we use the geometric induction in the case when it is an exact functor.

Consider the parabolic subalgebra
\[ \mathfrak{p}= \h\oplus\bigoplus_{2\leq i\leq n}\g_{\ep_1-\ep_i}\oplus\bigoplus_{2\leq i\neq j\leq n}\g_{\ep_i-\ep_j}.\]
Its Levi subalgebra $\mathfrak{l}$ is isomorphic to $\q(1)\oplus\q(n-1)$. Let
\[n_{\p} = \bigoplus_{2\leq i\leq n}\g_{\ep_1-\ep_i}\]
denote the nilpotent radical of~$\p$.

 Let $t$ be a positive integer. A dominant integral weight $\lambda\in\Lambda^+$ of $\q(n)$ is called \textit{$t$-admissible} if $\lambda=(t,\lambda_2,\dots,\lambda_n)$ such that $t+\lambda_i\neq 0$, in other words the first mark
 of $\lambda$ is $t$ and $\lambda$ is $\p$-typical.

 Let $\mathcal F_{\mathfrak{l}}(t)$ denote the category of finite-dimensional $\mathfrak{l}$-modules on which $H_1$ acts by $t$ and all weights of $\q(n-1)$ have integral marks strictly less than $t$. Let $\mathcal F^n(t)$ denote the Serre
 subcategory of $\mathcal F_n$ generated
 by $L(\lambda)$ for all $t$-admissible $\lambda$. Define the functors
\[ \Gamma^t\colon \ \mathcal F_{\mathfrak{l}}(t)\to \mathcal F^n(t), \qquad R^t\colon \ \mathcal F^n(t)\to \mathcal F_{\mathfrak{l}}(t)\]
 by
\[\Gamma^t(M):=\Gamma_0(G/P, M),\qquad R^t:=\operatorname{Ker}(H_1-t).\]
\begin{prop}\label{tadmissible} The functors $\Gamma^t$ and $R^t$ define an equivalence between $\mathcal F_{\mathfrak{l}}(t)$ and $\mathcal F^n(t)$.
\end{prop}

\begin{proof} By Proposition \ref{typicallemma}, $\Gamma_i(G/P,M)=0$ for $i>0$ and every $M\in \mathcal F_{\mathfrak{l}}(t)$. Furthermore, $\Gamma_0(G/P,M)$ is simple if $M$ is simple. On the other hand, $R^t(N)=H^0(n_\p,N)$ for any $N\in\mathcal F^n(t)$.
 That implies $\Gamma^t$ is left adjoint to $R^t$, both functors are exact and establish bijection on the sets of isomorphism classes of simple objects in both categories. Hence these functors provide an equivalence between the
 two categories.
 \end{proof}

\subsection[Reduction to q(n-1)]{Reduction to $\boldsymbol{\q(n-1)}$}\label{section5.2}

Note that every module in $\mathcal F_l(t)$ is of the form $L_{\q(1)}\boxtimes M$ for some $M\in\mathcal F^{n-1}$.

\begin{cor}\label{Standard block} Let $\lambda = (\lambda_1,\dots,\lambda_k, 1, 0,\dots,0,-\lambda_k,\dots,-\lambda_1)$, $\mu = (\mu_1,\dots,\mu_{k'},1,0,\dots,0,\linebreak-\mu_{k'},\dots,-\mu_1)$ be $\q(n)$ dominant weights in the standard block. For $t>>\max\{\lambda_1,\mu_1\}$,
\begin{align*}
\Ext_{\q(n)}^1(L(\lambda),L(\mu)) &= \Ext_{\q(n)}^1\big(L\big(t,\tilde{\lambda}\big),L(t,\tilde{\mu})\big)\\
&= \Ext_{\q(1)\oplus\q(n-1)}^1\big(L(t)\boxtimes L\big(\tilde{\lambda}\big),L(t)\boxtimes L(\tilde{\mu})\big)\\
&=\Ext_{\q(n-1)}^1\big(L\big(\tilde{\lambda}\big),L(\tilde{\mu})\big)\oplus \Ext_{\q(n-1)}^1\big(L\big(\tilde{\lambda}\big),\Pi L(\tilde{\mu})\big),
\end{align*}
where $\tilde{\lambda} = (\lambda_1-1,\dots,\lambda_k-1,0,\dots,0,1-\lambda_k,\dots,1-\lambda_1)$ and $\tilde{\mu}=(\mu_1-1,\dots,\mu_{k'}-1,0,\dots,0,\allowbreak 1-\mu_{k'},\dots,1-\mu_1)$. If $k=0$, then $\lambda = (1,0,\dots,0)$ and $\tilde{\lambda}=0$, and similarly for $k'=0$.
\end{cor}

\begin{proof}
 The first equality follows from \cite[Lemma~5.12]{Ser-ICM}, which shows there is an equivalence of categories between $\mathcal{F}^n_{(1,0,\dots,0)}$ (``standard block'') and $\mathcal{F}^n_{(t,0,\dots,0)}$ given by a
 composition of translation functors. Under this equivalence, $L(\lambda)$ maps to $L\big(t,\tilde\lambda\big)$. The second equality follows from Proposition~\ref{tadmissible}. The last equality follows from Lemma~\ref{eqcohomology}.
\end{proof}

Now, using the $\q(2)$ Ext quiver in \cite[Theorem~27]{Maz}, this corollary computes all extensions occurring in Theorem~\ref{MainTheorem}(4).

\begin{Remark}\label{Shapiro} In the standard block, the $\sq(3)$ extensions are the same as the $\q(3)$ extensions. Indeed, lemma \ref{Ind-Res} shows if $\lambda = (a,1,-a)$, $\mu = (b,1,-b)$, then $\Res L(\lambda) = L_{\sq}(\lambda)\oplus \Pi L_{\sq}(\lambda)$ and $\Ind L_{\sq}(\lambda)=L(\lambda)$. Now Shapiro's lemma implies
\begin{gather*} \Ext^1_{\q}(\Ind L_{\sq}(\lambda), L(\mu)) = \Ext^1_{\sq}(L_{\sq}(\lambda), L_{\sq}(\mu)\oplus \Pi L_{\sq}(\mu)),\qquad \text{and}\\
\Ext^1_{\q}(\Ind \Pi L_{\sq}(\lambda), L(\mu)) = \Ext^1_{\sq}(\Pi L_{\sq}(\lambda), L_{\sq}(\mu)\oplus \Pi L_{\sq}(\mu)).
\end{gather*}
Since $L(\lambda)=\Pi L(\lambda)$, we have $\Ext^1_{\q}(L(\lambda),L(\mu)) = \Ext^1_{\sq}(L_{\sq}(\lambda),L_{\sq}(\mu))$. Also $\Res L(1,0,0) = L_{\sq}(1,0,0)$, hence $\Ind L_{\sq}(\lambda) = \Ind\Pi L_{\sq}(\lambda)$ implies by Shapiro's lemma \[\Ext^1_{\q}(L(\lambda),L(1,0,0))=\Ext^1_{\sq}(L_{\sq}(\lambda),L_{\sq}(1,0,0))=\Ext^1_{\sq}(\Pi L_{\sq}(\lambda),L_{\sq}(1,0,0)).\]
\end{Remark}

\subsection{Relations}\label{section5.3}

\subsubsection[Relations for g=sq(3)]{Relations for $\boldsymbol{\g=\sq(3)}$}\label{section5.3.1}
All irreducible modules and projective covers considered here will be for $\sq(3)$, and we will omit the subscripts from $L_{\sq}(\lambda)$ and $P_{\sq}(\lambda)$. If $\lambda\neq(1,0,0)$, $\mu=(a,1,-a),a>1$, then $[P(\lambda):\mathcal{E}(\mu)]=[\mathcal{E}(\mu):L(\lambda)]$ by Theorem~\ref{BGG Reciprocity}. We note $L_{\sq}(a,1,-a)\neq\Pi L_{\sq}(a,1,-a)$. If $\lambda=(1,0,0)$, then $[P(\lambda):\mathcal{E}(\mu)]=2[\mathcal{E}(\mu):L(\lambda)]$.
Now, the character formula for $L(\lambda):\lambda=(\lambda_1,\lambda_2,\lambda_3)\in\Lambda^+$ is shown in \cite{PS1} to equal the generic character formula for all $\lambda\neq (1,0,0)$. This combined with character formula for $\mathcal{E}(\lambda)$ (\ref{Euler Characteristic}) implies
\begin{gather*}
\mathcal{E}(1,0,0) = 0; \hspace{.2cm} \mathcal{E}(2,1,-2) = [L(1,0,0)]+[L(2,1,-2)],\\
 \mathcal{E}(a,1,-a) = [L(a,1,-a)]+[L(a-1,1,-a+1)].
\end{gather*}

Thus, using (\ref{Projective multiplicity}), we find
\begin{gather*}
[P(1,0,0)] =2[L(1,0,0)]+2[ L(2,1,-2)],\\
[P(2,1,-2)] = [L(1,0,0)]+2[L(2,1,-2)]+[L(3,1,-3)],\\
[P(a,1,-a)] = [L(a-1,1,1-a)]+2[L(a,1,-a)]+[L(a+1,1,-a-1)]\qquad \text{for} \quad a>2.
\end{gather*}
This forces the radical filtrations for $P(\lambda)$ to be as shown in Appendix~\ref{appendixA}, since we know all possible extensions of simples, and hence $\operatorname{rad}P(\lambda)/\operatorname{rad}^2P(\lambda)$.

Let $V\in\mathcal{F}$ have radical filtration $V=\operatorname{rad}^0(V)\supset\operatorname{rad}^1 (V)\supset\cdots\supset \operatorname{rad}^k(V)=0$. Let $\operatorname{rad}_i=\operatorname{rad}^i/\operatorname{rad}^{i+1}$ and denote the radical filtration by
\[\operatorname{rad}_0V\big| \operatorname{rad}_1V\big|\cdots\big|\operatorname{rad}_{k-1}.\]
Let
\begin{gather*} a_1\in \Hom_{\mathbf{C}Q}(L(1,0,0),L(2,1,-2)) , \\ b_1\in \Hom_{\mathbf{C}Q}(L(2,1,-2)L(1,0,0)),\\ a_t\in\Hom_{\mathbf{C}Q}(L(t,1,-t), L(t+1,1,-t-1)),\\ b_t\in\Hom_{\mathbf{C}Q}(L(t+1,1,-t-1),L(t,1,-t))
\end{gather*} be paths on the quiver. We identify each $\gamma\in\Hom_{\mathbf{C}Q}(L(\lambda),L(\mu))$ with a corresponding element of $\Hom_{\g}(P(\lambda),\operatorname{rad}P(\mu)/\operatorname{rad}^2P(\mu))$ as in Proposition~\ref{relations}. Then using the radical filtrations in Section~\ref{section6.1},
\begin{gather*}
\begin{split}&
\im(a_1b_1) = \im(a_1)(L(2,1,-2)\big| L(1,0,0)) = L(2,1,-2)\qquad \text{and}\\
&\im(b_2a_2) = \im(b_2)(L(2,1,-2)\big| L(3,1,-3))=L(2,1,-2)
\end{split}
 \end{gather*}
 and hence $a_1b_1=b_2a_2$. Likewise we find $a_2a_1=b_1b_2=0$, $a_{t+1}a_t=b_tb_{t+1}=0$ and $b_ta_t=a_{t-1}b_{t-1}$ for $t\geq 3$. The computation for the $a_t'\in\Hom_Q(\Pi L(t+1,1,-t-1),\Pi L(t,1,1))$ and $b_t'\in\Hom_{Q}(\Pi L(t,1,-t),\Pi L(t+1,1,-t-1))$ is identical.

 \subsubsection{Translation functor from the principal to the standard block}\label{section5.3.2}

 Consider the translation functors:
\[ TV = \operatorname{pr}_1(V\otimes L(1,0,0))\qquad \text{and} \qquad T^*W = \operatorname{pr}_2(W\otimes L(0,0,-1)),\]
where $\operatorname{pr}_1$ is projection to standard block and $\operatorname{pr}_2$ is projection to principal block. It is well known $T$, $T^*$ are both exact and left and right adjoint to each other.

\begin{Lemma}\label{translations} We have $TL_\q(1,0,-1)=0$ and $TL_\q(a,0,-a)\cong L_\q(a,1,-a)$ for $a\geq 2$.
\end{Lemma}
\begin{proof} For the first assertion we use
\[\Hom_\q((L_\q(1,0,0), TL_\q(1,0,-1))=\Hom_\q(T^*L_\q(1,0,0), L_\q(1,0,-1))=0.\]
 Since all simple constituents of $TL_\q(1,0,-1)$ are isomorphic to $L_\q(1,0,0)$ by the weight argument, the statement follows.

 For the second assertion, use Lemma \ref{induction functor properties}(2) to check that \[ T\Gamma_0(G/B,v(a,0,-a))=\Gamma_0(G/B,v(a,1,-a)).\]
 Information about the multiplicities of simple modules in
 \[ \Gamma_0(G/B,v(a,0,-a)) \qquad \text{and}\qquad \Gamma_0(G/B,v(a,1,-a))\]
 allows to conclude that
\[ T(\operatorname{Top}\Gamma_0(G/B,v(a,0,-a)))=\operatorname{Top}\Gamma_0(G/B,v(a,1,-a)).\tag*{\qed}
\]\renewcommand{\qed}{}
 \end{proof}

\subsubsection[Relations for g=q(3)]{Relations for $\boldsymbol{\g=\mathfrak{q}(3)}$}\label{section5.3.3} As before, we use BGG reciprocity to find
\begin{gather*}
[P(1,0,0)] =4[L(1,0,0)]+2[ L(2,1,-2)],\\
[P(2,1,-2)] = 2[L(1,0,0)]+2[L(2,1,-2)]+[L(3,1,-3)],\\
[P(a,1,-a)] = [L(a-1,1,1-a)]+2[L(a,1,-a)]+[L(a+1,1,-a-1)]\qquad \text{for}\quad a>2.
\end{gather*}

Recall functors $\Res$, $\Ind$ from Section~\ref{generalind}.
It is easy to verify the following isomorphisms
\[ P_{\q}(\lambda)\cong \Ind P_{\mathfrak{sq}}(\lambda), \qquad \Res P_{\q}(\lambda)\cong P_{\mathfrak{sq}}(\lambda)\oplus \Pi P_{\mathfrak{sq}}(\lambda).\]
Furthermore, we have an isomorphism $P_{\mathfrak{sq}}(\lambda)\cong \Pi P_{\mathfrak{sq}}(\lambda)$ only for $\lambda=(1,0,0)$.
Moreover, from relations for $\mathfrak{sq}(3)$ we know
\[\Hom_{\mathfrak{sq}}(P_{\mathfrak{sq}}(\lambda),\Pi P_{\mathfrak{sq}}(\mu))=0\]
if both $\lambda$ and $\mu$ are not equal to $(1,0,0)$. This implies
\[\Ind \colon \Hom_{\mathfrak{sq}}(P_{\mathfrak{sq}}(\lambda), P_{\mathfrak{sq}}(\mu))\to\Hom_{\mathfrak{q}}(P_{\mathfrak{q}}(\lambda),P_{\mathfrak{q}}(\mu))\]
is an isomorphism if $\lambda,\mu\neq (1,0,0)$. Therefore we have unique lift of the arrows $a_t$, $b_t$ to $\tilde a_t$, $\tilde b_t$ for $t\geq 2$ and the relations between them are the same as for $\mathfrak{sq}(3)$.
For the arrows $a_1\colon P_{\mathfrak {sq}}(1,0,0)\to P_{\mathfrak {sq}}(2,1,-2)$ and $b_1\colon P_{\mathfrak {sq}}(2,1,-2)\to P_{\mathfrak {sq}}(1,0,0)$ we define
\[\tilde a_1=\Ind (a_1+\Pi b_1),\qquad \tilde b_1=\Ind (b_1+\Pi a_1).\]
Then we immediately obtain $\tilde a_1\tilde b_1=0$ and $\tilde a_2\tilde a_1=\tilde b_1 \tilde b_2=0$ from the relations for $\mathfrak{sq}(3)$.
From multiplicities of simple modules in projectives and information about $\Ext^1_{\q}$, one can compute
the layers of the radical filtration in projective modules which are listed in Section~\ref{section6.2}.

There is one additional loop arrow $h\colon P_{\mathfrak {q}}(1,0,0)\to P_{\mathfrak {q}}(1,0,0)$.
Let us show that we can choose $h$ in such a way that $h^2=0$. We claim that there exists a $\q(3)$-module $R$ with radical filtration
\[L_{\q}(1,0,0)\,|\, L_{\q}(2,1,-2)\,|\, L_{\q}(1,0,0).\]
Indeed, it is proven independently in Section~\ref{section5} that there exists a $\q$-module $R'$ with the radical filtration
\[\Pi\C\,|\, L_{\q}(2,0,-2)\,|\, \C .\]
Set $R=TR'$ where $T$ is the translation functor from the principal to the standard block. Lemma~\ref{translations} ensures that the composition factors of~$R$ are as desired.
Furthermore, $R'$~is a~quotient of $P_\q(0)$ and hence~$R$ is a quotient of $P_\q(1,0,0)$. That settles the radical filtration of~$R$.

Now we can use an exact sequence
\[0\to R\to P_\q(1,0,0)\to R\to 0\]
to construct $h\in\End_{\q}(P_{\mathfrak {q}}(1,0,0))$ with image and kernel isomorphic to $R$. In fact from the radical filtration of~$R$ we know that $\Res R\cong P_{\mathfrak {sq}}(1,0,0)$.

Now we can use Lemma~\ref{general} with $M=R$, $P_\q(1,0,0)=\Ind R$ and $\theta=h$.
The algebra $\End_q(P_{\q}(1,0,0))$ is generated~$h$ and $u:=\tilde b_1\tilde a_1$ with $u^2=0$.
Lemma~\ref{general} implies $hu=uh$.

\subsection{``Half-standard'' block}\label{section5.4}

We can compute the Ext quiver for $\q(3)$ half-standard block by just computing the radical filtrations of the projective covers. The character formula for $L(\lambda)$ when $\text{\rm wt}(\lambda)=\delta_{\frac{3}{2}}$ is the same as the generic character formula. Hence, we find
\begin{gather*}
\mathcal{E}\big(\tfrac{3}{2},\tfrac{1}{2},-\tfrac{1}{2}\big) = \big[L\big(\tfrac{3}{2},\tfrac{1}{2},-\tfrac{1}{2}\big)\big],\qquad \mathcal{E}\big(\tfrac{5}{2},\tfrac{3}{2},-\tfrac{5}{2}\big) = \big[L\big(\tfrac{5}{2},\tfrac{3}{2},-\tfrac{5}{2}\big)\big]+\big[L\big(\tfrac{3}{2},\tfrac{1}{2},-\tfrac{1}{2}\big)\big],\\
 \mathcal{E}\big(\tfrac{2a+1}{2}, \tfrac{3}{2},\tfrac{-2a-1}{2}\big)=\big[L\big(\tfrac{2a+1}{2}, \tfrac{3}{2},\tfrac{-2a-1}{2}\big)\big]+\big[L\big(\tfrac{2a-1}{2}, \tfrac{3}{2},\tfrac{-2a+1}{2}\big)\big], \qquad \text{for} \quad a>2.
\end{gather*}

Now by our BGG reciprocity result, $[P(\lambda):\mathcal{E}(\mu)] = [\mathcal{E}(\mu):L(\lambda)]$, hence
\begin{gather*}
\big[P\big(\tfrac{3}{2},\tfrac{1}{2},-\tfrac{1}{2}\big)\big] = 2\big[L\big( \tfrac{3}{2},\tfrac{1}{2},-\tfrac{1}{2}\big)\big]+\big[L\big(\tfrac{5}{2},\tfrac{3}{2},-\tfrac{5}{2}\big)\big],\\
\big[P\big(\tfrac{5}{2},\tfrac{3}{2},-\tfrac{5}{2}\big)\big] = \big[L\big(\tfrac{3}{2},\tfrac{1}{2},-\tfrac{1}{2}\big)\big] + 2\big[L\big(\tfrac{5}{2},\tfrac{3}{2},-\tfrac{5}{2}\big)\big] + \big[L\big(\tfrac{7}{2},\tfrac{3}{2},-\tfrac{7}{2}\big)\big],\\
\big[P\big(\tfrac{2a+1}{2},\tfrac{3}{2},-\tfrac{2a+1}{2}\big)\big] = \big[L\big(\tfrac{2a-1}{2},\tfrac{3}{2},-\tfrac{2a-1}{2}\big)\big] + 2\big[L\big(\tfrac{2a+1}{2},\tfrac{3}{2},-\tfrac{2a+1}{2}\big)\big] \\
\hphantom{\big[P\big(\tfrac{2a+1}{2},\tfrac{3}{2},-\tfrac{2a+1}{2}\big)\big] = }{} + \big[L\big(\tfrac{2a+3}{2},\tfrac{3}{2},-\tfrac{2a+3}{2}\big)\big],\qquad a>2.
\end{gather*}
Using these composition factors and that there are no self extensions, we find the radical filtration for $P\big(\tfrac{3}{2},\tfrac{1}{2},-\tfrac{1}{2}\big)$. Then using $\Ext^1_{\q(3)}(L(\lambda),L(\mu))=\Ext^1_{\q(3)}(L(\mu),L(\lambda))$ when $\text{\rm wt}(\lambda)=\text{\rm wt}(\mu)\allowbreak =\delta_{\frac{3}{2}}$, we inductively (on $a$) find the radical filtrations for $P\big(\tfrac{2a+1}{2},\tfrac{3}{2},-\tfrac{2a+1}{2}\big)$. This computes the possible extensions, and moreover determines the relations as written in the theorem. Finally, the~$\mathfrak{sq}(3)$ half-standard block follows from Shapiro's lemma (see Remark~\ref{Shapiro}).

\begin{Remark} It was shown in \cite{BD} that half-integral blocks in $\mathcal{F}_{\q(n)}$ of atypicality~$r$ are equivalent to modules over the Khovanov are algebra $ K_r^{+\infty}$. In our case, $r=1$ and the arc algebra is the zig-zag algebra of the semi-infinite linear quiver.
\end{Remark}

\section{Principal block}\label{section6}
\subsection[Ext quiver for the principal block for sq(3)]{Ext quiver for the principal block for $\boldsymbol{\mathfrak{sq}(3)}$}\label{section6.1}

In this subsection we use the notation
\[\g=\mathfrak{sq}(3),\qquad L(a)=L_{\mathfrak{sq}(3)}(a,0,-a),\qquad P(a)=P_{\mathfrak{sq}(3)}(a,0,-a).\]
We start with the following
\begin{Lemma}\label{symmetrysq} If $a>0$ then $L(a)^*\cong \Pi L(a)$ and $P(a)^*\cong \Pi P(a)$. Furthermore, $P(0)\cong P(0)^*$.
\end{Lemma}
\begin{proof} If the highest weight $\h$-module of $L(a)$ is $v(a,0,-a)$ then the highest weight $\h$-module of $L(a)^*$ is $v(-a,0,a)^*$~\cite{Fri}. The first assertion follows from the isomorphism
 $v(-a,0,a)^*\cong \Pi v(a,0,-a)$ of $\h$-modules when $a>0$. If $I(L),P(L)$ denote the injective, projective hull of a simple $\g$-mod $L$ respectively, then by~\cite{Ser1}
\[I(L)\cong P(L)\otimes T,\]
 where $T\cong S^{\rm top}(\g_{\bar 1})$. In our case $S^{\rm top}(\g_{\bar 1})=S^8(\g_{\bar 1})$ is the trivial $\g_{\bar 0}$-module. Hence we have $I(L)\cong P(L)$. On the other hand $P(L)^*\cong I(L^*)$. Hence if $a>0$, $P(L(a))^*\cong P(L(a)^*)\cong\Pi P(L(a))$. If $a=0$, then $L(0)^*= L(0)$ and consequently $P(L(0))^*\cong P(L(0)^*)\cong P(L(0))$.
\end{proof}

\begin{cor}\label{extsymmetry} If $a,b>0$, then
\[ \Ext^1_\g(L(a),L(b))\cong \Ext^1_\g(L(b),L(a)),\qquad \Ext^1_\g(L(a),L(0))\cong \Ext^1_\g(\Pi L(0),L(a)).\]
 \end{cor}

Note, in the case $\g=\mathfrak{sq}(3)$ and $\mu=(a,0,-a),a>0$, we have $[P(\lambda):\mathcal{E}(\mu)]=[\mathcal{E}(\mu):L(\lambda)]$ by Theorem \ref{BGG Reciprocity}. Now, using the character formula for $L(a,0,-a)$ \cite{PS2} and $\mathcal{E}(a,0,-a)$~(\ref{Euler Characteristic}), we find
\begin{gather*} \mathcal{E}(0) = 0,\qquad \mathcal{E}(1,0,-1) = [L(1)],\qquad \mathcal{E}(2,0,-2) = [L(2)]+[L(1)]+2[\C],\\
 \mathcal{E}(a,0,-a)=[L(a)]+[L(a-1)]\qquad \text{if} \quad a>2.
\end{gather*}

Thus, using \ref{Projective multiplicity}, we find
\begin{gather*}
[P(0)] = 4[\C]+2[L(1)]+2[ L(2)],\\
[P(1)] =2[\C]+2[L(1)]+[L(2)],\\
[P(2)] = 2[\C]+[L(1)]+2[L(2)]+[L(3)],\\
[P(a)] = [L(a-1)]+2[L(a)]+[L(a+1)]\qquad \text{for} \quad a>2.
\end{gather*}

Furthermore, it follows from \cite{PS2} that $\Gamma_i(G/B,v(a,0,-a))=0$ if $a\geq 2$ and $i\geq 1$.
Therefore if $a\geq 3$, we obtain a non-split exact sequence
\[ 0\to L(a-1)\to \Gamma_0(G/B,v(a,0,-a))\to L(a)\to 0,\]
which gives a non-trivial extension $\Ext^1_\g(L(a),L(a-1))$. There are no more by Proposition~\ref{m^0 greater than Ext}.
\begin{Lemma}\label{generalext} Let $a\geq 3$, then
\begin{gather*} \Ext^1_\g(L(a),L(a-1))=\Ext^1_\g(L(a-1),L(a))=\C\end{gather*} and $\Ext^1_\g(L(a),L)=0$
 for all simple $L$ not isomorphic to $L(a-1)$. \end{Lemma}
\begin{proof} We do not have self-extensions by Theorem~\ref{selfextensions2}. If $b<a$, then a non-trivial extension of~$L(a)$ by $L(b)$ or~$\Pi L(b)$ is a quotient of $\Gamma_0(G/B,v(a,0,-a))$ by Proposition~\ref{induction functor properties}(3).
 That forces $b=a-1$ and also implies $\Ext^1_\g(L(a),\Pi L(b))=0$. The case $b>a$ follows by Corollary~\ref{extsymmetry}.
 \end{proof}

\begin{Remark} Our choice of labelling $\Pi L(a)$ vs $L(a)$ for $a\geq 2$ is determined by the above lemma. For $a=1$ we assume that $L(1)$ is dual to the adjoint representation in $\mathfrak{psq}(3)$. For $a=2$ the choice will be
 clear from the following lemma. Note that in the same way using
 \[[\Gamma_0(v(2,0,-2)):L(1)]_\Pi=1,\] we obtain
 \begin{equation}\label{ineq}
 \dim\Ext^1_{\g}(L(2),L(1)\oplus \Pi L(1))=\dim\Ext_{\g}^1(L(1)\oplus \Pi L(1),L(2))\leq 1.
 \end{equation}
 \end{Remark}

\begin{Lemma}\label{specialext}
\begin{gather}
\Ext_{\g}^1(L(1),\C)=\C\qquad \text{and}\qquad \Ext_{\g}^1(\Pi L(1),\C)=0, \label{spe1}\\
\Ext_{\g}^1(L(2),\C)=0\qquad \text{and}\qquad \Ext_{\g}^1(\Pi L(2),\C)=\C ,\label{spe2}\\
\Ext_{\g}^1(L(1), L(2))=0\qquad \text{and}\qquad\Ext_{\g}^1(\Pi L(1),L(2))=0.\label{spe3}
\end{gather}
 \end{Lemma}
 \begin{proof} Identify $\Pi L(1)$ with the simple Lie superalgebra $\mathfrak{psq}(3)$, then $\operatorname{Der}\mathfrak{psq}(3)=\Pi \C$~\cite{Kac}. This implies~(\ref{spe1}) by use of duality and
 Lemma~\ref{symmetrysq}.

In order to prove remaining identities we consider the projective module $P(0)$. We know all its simple constituents: $L(1)$, $\Pi L(1)$, $L(2)$, $\Pi L(2)$ and $L(0)$, $\Pi L(0)$, the last two appear with multiplicity~$2$. Since~$P(0)$ is projective, we know its super dimension is~0, so $L(0)$ and $\Pi L(0)$ occur with same multiplicity.
 Assume for the sake of contradiction that
 \[\Ext_{\g}^1(L(2),\C)=\Ext_{\g}^1(\Pi L(2),\C)=0.\]
 That would imply that
\[ P(0)/\operatorname{rad}P(0)=L(0),\qquad \operatorname{rad}P(0)/\operatorname{rad}^2P(0)=L(1).\]
 Furthermore, since $\q(3)^*$ is a quotient of $P(0)$ we know that $\operatorname{rad}^2P(0)/\operatorname{rad}^3P(0)$ contains $\Pi L(0)$. But it must contain $L(2)$ or $\Pi L(2)$ (otherwise~$L(2)$ will not appear in~$P(0)$).
 We have the inequality
\[\big[\operatorname{rad}^2P(0)/\operatorname{rad}^3P(0):L\big]\leq \dim\Ext^1_\g(L(1),L)\]
 for any simple $L$. Therefore $\operatorname{rad}^2P(0)/\operatorname{rad}^3P(0)$ contains only one copy of $\Pi L(0)$ and one copy of either $L(2)$ or $\Pi L(2)$, \eqref{ineq}.
 Without loss of generality we may assume
\[\operatorname{rad}^2P(0)/\operatorname{rad}^3P(0)=\Pi L(0)\oplus \Pi L(2).\]
 Let $M=P(0)/\operatorname{rad}^3P(0)$. Then by our assumption we have $M^*\subset P(0)^*\cong P(0)$, and we have the exact sequence
\[0\to M^*\to P(0)\to M\to 0.\]
 Furthermore, it also follows from our assumptions that the radical and socle filtrations on~$M$ are the same. In particular, it follows that $M^*/\operatorname{rad} M^*=\Pi L(0)\oplus L(2)$. That would imply
 $\Ext_\g^1(\Pi L(2)\oplus \Pi L(0),L(2)\oplus \Pi L(0))\neq 0$. But that contradicts our original assumption.

 The above argument implies that $\Ext^1_\g(\C,\Pi L(2))=\C$. Therefore
\begin{gather*}
\operatorname{rad}P(0)/\operatorname{rad}^2P(0)=L(1)\oplus\Pi L(2) \qquad\text{or}\\ \operatorname{rad}P(0)/\operatorname{rad}^2P(0)=L(1)\oplus\Pi L(2)\oplus L(2).\end{gather*}
 However, it is easy to see that the latter case is impossible since otherwise by self-duality of $P(0)$ we have $\operatorname{soc}^2P(0)/\operatorname{soc}P(0) =\Pi L(1)\oplus\Pi L(2)\oplus L(2)$ and that would imply $[P(0):L(2)]_\Pi>2$.
 Therefore we have $ \operatorname{rad}P(0)/\operatorname{rad}^2P(0)=L(1)\oplus\Pi L(2)$, and that implies~(\ref{spe2}).

 Moreover, we obtain the radical filtration of $P(0)$ as shown in Appendix~\ref{appendixA}. Since $\Pi v(2,0,-2)$ is the highest weight $\h$-submodule of $P(0)$, we have a homomorphism
 $\gamma\colon \Gamma_0(G/B,\Pi v(2,0,-2))\allowbreak \to P(0)$ and from the
 socle filtration of $P(0)$ (in this case we have $\operatorname{soc}^kP(0)=\operatorname{rad}^{5-k}P(0)$) we obtain that $\gamma$ is injective. The socle filtration of $\Gamma_0(G/B,\Pi v(2,0,-2))$ is inherited from that of~$P(0)$. We get
 \begin{gather*}
 \operatorname{soc}\Gamma_0(G/B,\Pi v(2,0,-2))=L(0),\\
 \operatorname{soc}^2\Gamma_0(G/B,\Pi v(2,0,-2))/\operatorname{soc}\Gamma_0(G/B,\Pi v(2,0,-2))=\Pi L(1),\\
 \operatorname{soc}^3\Gamma_0(G/B,\Pi v(2,0,-2))/\operatorname{soc}^2\Gamma_0(G/B,\Pi v(2,0,-2))=\Pi L(0),\\
 \operatorname{soc}^4\Gamma_0(G/B,\Pi v(2,0,-2))/\operatorname{soc}^3\Gamma_0(G/B,\Pi v(2,0,-2))=\Pi L(2).
 \end{gather*}
 That proves (\ref{spe3}).
 \end{proof}

Note that Lemmas \ref{generalext} and~\ref{specialext} prove that the Ext-quiver for the principal block for $\mathfrak{sq}(3)$ coincides with the one in Theorem~\ref{MainTheorem sq}.

\subsection[Relations for the principal block for g=sq(3)]{Relations for the principal block for $\boldsymbol{\g=\mathfrak{sq}(3)}$}\label{section6.2}

We first compute the radical filtration of all indecomposable projectives. Using the self-duality (up to parity) of $P(a)$ and fact that we know all possible extensions of simples, we automatically know the top~2 and bottom~2 layers. It turns out the other layers are fixed as well, as shown below. Diagrams are in Appendix~\ref{appendixA}.
For $P(0)$, we just obtained in the proof of Lemma~\ref{specialext}.

For $P(1)$, $\operatorname{Top}P(1)=L(1)$ and $\operatorname{rad}P(1)/\operatorname{rad}^2P(1)=\Pi\C$. The only extension with $\Pi\C$ is~$L(2)$.

For $P(2)$, $\operatorname{Top}P(2)=L(2)$ and $\operatorname{rad}P(2)/\operatorname{rad}^2P(2)= L(0)+L(3)$. Considering possible extensions of $L(0)$ and $L(3)$, we find $L(1)$ is subquotient of $\operatorname{rad}^2P(2)$ and $\operatorname{soc}P(2)=L(2)$. Finally, there only exists an extension $\frac{L(1)}{\Pi\C}$ and not~$\frac{L(1)}{\C}$.

For $P(a)$, $a\geq3$, $\operatorname{Top}P(a)=L(a)$ and $\operatorname{rad}P(a)/\operatorname{rad}^2P(a)= L(a-1)+L(a+1)$.

Now, we will use the radical filtrations to compute all relations. Note in all cases,
\[\dim\Hom_{\g}(P(a,0,-a), M)=[M:L(a,0,-a)].\]

Let
\begin{gather*} a\in\Hom_{\g}(P(1),P(0)),\qquad c\in\Hom_{\g}(P(0),\Pi P(1)),\\
d\in\Hom_{\g}(P(2),\Pi P(0)), \qquad b\in\Hom_{\g}( P(0), P(2)),
\end{gather*} and the $a'$, $b'$, $c'$, $d'$ be the corresponding parity-changed arrows. Since $[P(1):\Pi L(1)] = [P(2):\Pi L(2)]=0$, we obtain $bd' = ca = b'd = c'a' = 0$.

Let
\[ a_t\in\Hom_{\g}(P(t,P(t+1)), \qquad b_t\in\Hom_{\g}(P(t+1),P(t)). \]
Then
\[ [\Pi P(0): L(3)]= [P(3):L(0)] = 0, \qquad [P(t+1):L(t-1)]= [P(t-1):L(t+1)]=0\] for $t\geq 3$ implies the relations $db_2=a_2b=0$ and $a_{t+1}a_t = b_{t+1}b_t = 0$, $t\geq2$. By symmetry, we get the parity-changed analogues of these relations.

Next consider the cycle paths in $P(0)$: $ac'a'c$, $d'b'a'c$, $d'b'db$, $ac'db$ We know from previous paragraph that $ac'a'c=0=db'db$. Also, since $\dim\Hom_{\g}(P(0),\Pi P(0))=2$, $a'c$ and $db$ are not scalar multiples. Thus
\[\im(d'b'a'c) = L(0) = \im(ac'db)\]
and $\dim\End_{\g}(P(0))=2$, implies $d'b'a'c=\lambda_0 ac'db$ for $\lambda_0\in\C^*$. Similarly we find $dbac' = \lambda_0'a'cd'b'$.

Next consider the nontrivial cycle paths in $P(2)$: $b_2a_2$, $bac'd$. Since
\[\im(b_2a_2) = L(2) = \im(bac'd),\] and $\dim\End_{\g}(P(2))=2$, we conclude $b_2a_2=\lambda_2bac'd$, $\lambda_2\in\C^*$.

Finally, for $t\geq 3$, we find cycle paths in $P(a)$ are $a_tb_t$, $b_{t+1}a_{t+1}$. Both have image $L(t)$, hence $a_tb_t=\lambda_tb_{t+1}a_{t+1}$. Observe we can sufficiently scale all arrows and hence normalize all $\lambda_t,\lambda_t'\in\C^*$ to equal~1. The remaining $\dim\Hom_{\g}(P(a),P(b))$ calculations shows there are no other relations.

\subsection[The principal block of q(3)]{The principal block of $\boldsymbol{\q(3)}$}\label{section6.3}
We start with the following general statement.
\begin{Lemma}\label{extq}
Let $\lambda$, $\mu$ be two distinct weights in the principal block such that there exists $i,j\colon \lambda_i=\mu_j= 0$. If \[ \dim\Ext^1_{\mathfrak{sq}(n)}(L_{\mathfrak{sq}}(\lambda),L_{\mathfrak{sq}}(\mu)) +\dim\Ext^1_{\mathfrak{sq}(n)}(L_{\mathfrak{sq}}(\lambda),\Pi L_{\mathfrak{sq}}(\mu))\leq 1,
\] then
\[\Ext^1_{\mathfrak{sq}(n)}(L_{\mathfrak{sq}}(\lambda),L_{\mathfrak{sq}}(\mu))=\Ext^1_{\q(n)}(L(\lambda),L(\mu)).\]
\end{Lemma}
\begin{proof} By Corollary~\ref{Ind-Res}, $\Res L(\lambda) = L_{\sq}(\lambda)$ and $\Ind L_{\sq}(\lambda) = \frac{\Pi L(\lambda)}{L(\lambda)}$. The nonsplit short exact sequence $0\rightarrow L(\lambda)\rightarrow V\rightarrow \Pi L(\lambda)\rightarrow 0$ of $\q(n)$-modules, which exists by Theorem~\ref{selfextensions}, gives rise to long exact sequences ($\mu\neq\lambda$)
\begin{align*}
0&\leftarrow \Ext_{\q(n)}^1(L(\lambda),L(\mu))\leftarrow \Ext_{\q(n)}^1(V,L(\mu))\leftarrow \Ext_{\q(n)}^1(\Pi L(\lambda),L(\mu))\\
& \leftarrow\Ext_{\q(n)}^2(L(\lambda),L(\mu))\cdots,\\
0&\leftarrow \Ext_{\q(n)}^1(L(\lambda),\Pi L(\mu))\leftarrow \Ext_{\q(n)}^1(V,\Pi L(\mu))\leftarrow \Ext_{\q(n)}^1(\Pi L(\lambda),\Pi L(\mu))\\
& \leftarrow\Ext_{\q(n)}^2(L(\lambda),\Pi L(\mu))\cdots.
\end{align*}
Now the lemma follows from Shapiro's lemma, $\Ext_{\q(n)}^1(V,L(\mu)) = \Ext_{\mathfrak{sq}(n)}^1(L_{\mathfrak{sq}}(\lambda),L_{\mathfrak{sq}}(\mu))$, and the hypotheses.
\end{proof}

Lemma \ref{extq} implies that the Ext quiver for the principal block of $\q(3)$ is obtained from that of $\mathfrak{sq}(3)$ by adding a single arrow between each $L(a)$ and $\Pi L(a)$ (Theorem \ref{selfextensions}).

\begin{Lemma}\label{q relations} Let $\g=\q(3)$ and $\theta\in\Hom_{\q}(P(\lambda),\Pi P(\lambda))$ be the unique self extension for each $\lambda$ in the principal $\q$-block. Let
 $\mathbb P_{\mathfrak{sq}}$ $($resp., $\mathbb P_{\mathfrak{q}})$ denote the direct sum of all indecomposable projectives in the principal block for $\mathfrak{sq}(3)$ $($resp., $\q(3))$;
 $\mathcal{A}=\End_{\mathfrak{sq}}(\mathbb P_{\mathfrak{sq}})$ denote the algebra defined by the quiver with relations in principal $\mathfrak{sq}(3)$-block. Then the algebra defined by quiver with relations in principal $\q(3)$-block is
\[\mathcal{A}'=\End_{\q}(\mathbb P_{\mathfrak{q}})\cong\mathcal{A}\otimes\C[\theta]/\big(\theta^2\big).\]
\end{Lemma}

\begin{proof} Recall the functors $\Res$ and $\Ind$ from Section~\ref{generalind}. We would like to show that $M=\mathbb{P}_{\mathfrak{sq}}$ satisfies the assumptions of Lemma~\ref{general}.
 Corollary~\ref{Ind-Res} implies $\mathbb P_{\q}\cong\Ind M$. It remains to show the existence of $\theta$.
 The BGG reciprocity implies for $\lambda$, $\mu$ in principal block, $[P_{\q}(\lambda):L(\mu)]=2[P_{\mathfrak{sq}}(\lambda): L_{\mathfrak{sq}}(\mu)]$.
 Since $\Ind P_{\mathfrak{sq}}(\lambda)$ is projective,
\[ \Ind P_{\mathfrak{sq}}(\lambda)=P_{\q}(\lambda) \qquad \text{and}\qquad \Res P_{\q}(\lambda)=P_{\mathfrak{sq}}(\lambda)\oplus\Pi P_{\mathfrak{sq}}(\lambda).\]
If $\lambda\neq 0$ then $\Hom_{\mathfrak{sq}}(\Pi P_{\mathfrak{sq}}(\lambda),P_{\mathfrak{sq}}(\lambda))=0$ and therefore the above decomposition is unique.

By Frobenius reciprocity and fact that $\Hom_{\mathfrak{sq}}(P_{\mathfrak{sq}}(\lambda), \Pi P_{\mathfrak{sq}}(\lambda))=0$ for $\lambda\neq 0$, we have for $\lambda\neq 0$
\begin{align*}
\Hom_{\q}(\Ind P_{\mathfrak{sq}}(\lambda),\Ind \Pi P_{\mathfrak{sq}}(\lambda)) &= \Hom_{\mathfrak{sq}}(P_{\mathfrak{sq}}(\lambda), \Res \Ind \Pi P_{\mathfrak{sq}}(\lambda))\\
&= \Hom_{\mathfrak{sq}}(P_{\mathfrak{sq}}(\lambda), P_{\mathfrak{sq}}(\lambda)\oplus\Pi P_{\mathfrak{\mathfrak{sq}}}(\lambda))\\
&=\Hom_{\mathfrak{sq}}(P_{\mathfrak{sq}}(\lambda), \Res \Ind P_{\mathfrak{sq}}(\lambda))=\C.
\end{align*}

We choose $\theta_{\lambda}\colon P_{\mathfrak{q}}(\lambda)\to\Pi P_{\mathfrak{q}}(\lambda)$ corresponding to the identity map in $\Hom_{\mathfrak{sq}}(P_{\mathfrak{sq}}(\lambda),P_{\mathfrak{sq}}(\lambda))$ and set
$\bar P(\lambda)=\operatorname{Im}\theta$. Then we have an exact sequence
\[ 0\to\Pi\bar P(\lambda)\to P_{\q}(\lambda)\to \bar P(\lambda)\to 0,\]
with $\Res \bar P(\lambda)\cong P_{\mathfrak{sq}}(\lambda)$.
Now let us prove that for $\lambda=0$ we also have $\theta\colon P_{\q}(0)\to \Pi P_{\q}(0)$ with $(\Pi\theta)\theta=0$ and hence the exact sequence
\[ 0\to\Pi\bar P(0)\to P_{\q}(\lambda)\to \bar P(0)\to 0.\]
We use that $P_{\q}(0)=\operatorname{pr}\big(\Ind_{\q(3)_{\bar 0}}^{\q(3)}\C\big)$ where $\operatorname{pr}$ denote the projection on the principal block. Let $\mathfrak l=\q(3)_{\bar 0}\oplus \C\bar H$ where $\bar H=\bar H_1+\bar H_2+\bar H_3$.
Since $\bar H^2$ acts by zero on the modules of our block we have an exact sequence of $\mathfrak{l}$-modules
\[ 0\to\Pi\C\to \Ind_{\q(3)_{\bar 0}}^{\mathfrak l}\C\to \C\to 0,\]
and therefore the exact sequences
\begin{gather*}
0\to\Pi\Ind^{\q(3)}_{\mathfrak l}\C\xrightarrow{\alpha} \Ind_{\q(3)_{\bar 0}}^{\q(3)}\C\xrightarrow{\beta} \Ind^{\q(3)}_{\mathfrak l}\C\to 0,\\
0\to\Ind^{\q(3)}_{\mathfrak l}\C\xrightarrow{\Pi\alpha} \Pi\Ind_{\q(3)_{\bar 0}}^{\q(3)}\C\xrightarrow{\Pi\beta} \Pi\Ind^{\q(3)}_{\mathfrak l}\C\to 0.\end{gather*}
By setting $\theta=\operatorname{pr}(\Pi\alpha\beta)\operatorname{pr}$ we obtain the desired claim.

To finish the proof we just use Lemma~\ref{general}.
\end{proof}

\appendix

\section[Radical filtrations for P{g}(lambda) when g=sq(3), q(3)]{Radical filtrations for $\boldsymbol{P_{\g}(\lambda)}$ when $\boldsymbol{\g=\mathfrak{sq}(3), \q(3)}$}\label{appendixA}

In all radical filtrations, an edge denotes an extension. Observe for $P_{\q(3)}(a,0,-a)$, the ``left half'' corresponds to $\ker\theta$ and the ``right half'' corresponds to $\im\theta$.

\subsection[g=sq(3)]{$\boldsymbol{\g=\mathfrak{sq}(3)}$} \label{sq filtrations}

The radical filtrations are $(a\geq 3)$:
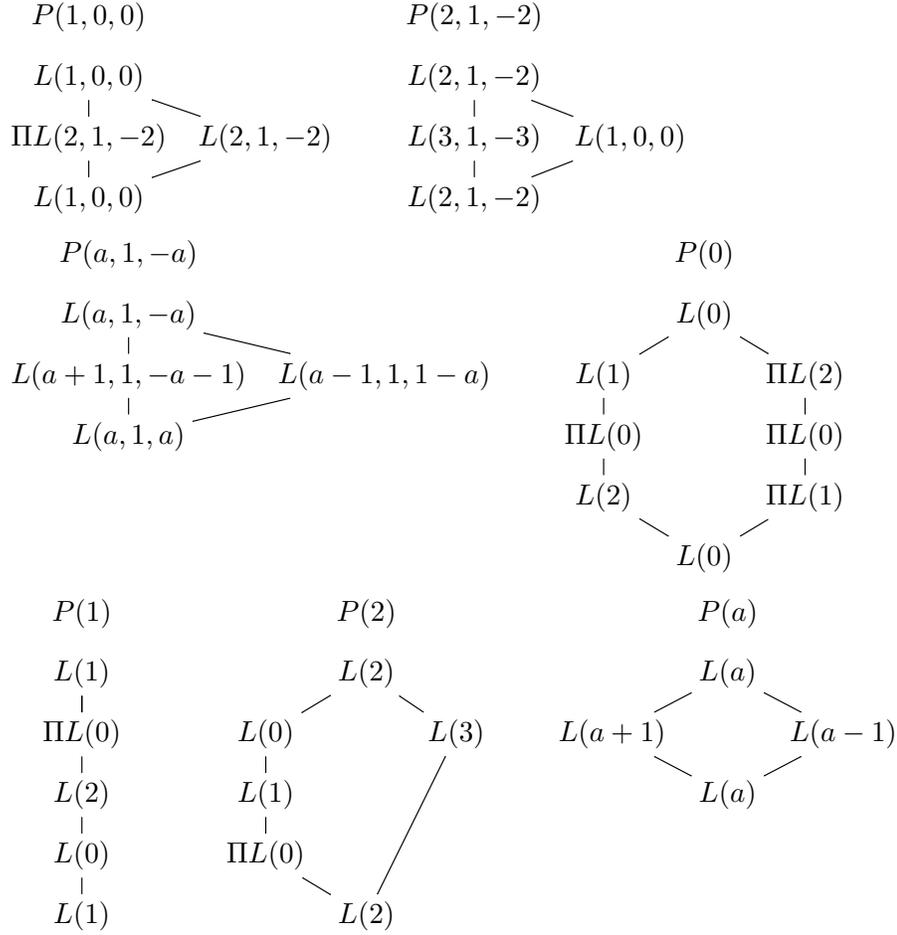
\begin{gather*}
\xymatrix@-1.5pc{P(1,0,0)&\\
L(1,0,0)\ar@{-}[d]\ar@{-}[dr]&\\
\Pi L(2,1,-2)\ar@{-}[d]&L(2,1,-2)\ar@{-}[dl]\\
L(1,0,0)&
}\qquad
\xymatrix@-1.5pc{P(2,1,-2)&\\
L(2,1,-2)\ar@{-}[d]\ar@{-}[dr]&\\
 L(3,1,-3)\ar@{-}[d]&L(1,0,0)\ar@{-}[dl]\\
L(2,1,-2)&\\
}
\\
\xymatrix@-1.5pc{P(a,1,-a)&\\
L(a,1,-a)\ar@{-}[d]\ar@{-}[dr]&\\
L(a+1,1,-a-1)\ar@{-}[d]&L(a-1,1,1-a)\ar@{-}[dl]\\
L(a,1,a)&\\
}
\qquad
\xymatrix@-1.5pc{&P(0)&\\
&L(0)\ar@{-}[dr]\ar@{-}[dl]&\\
 L(1)\ar@{-}[d]&&\Pi L(2)\ar@{-}[d]\\
\Pi L(0)\ar@{-}[d]&&\Pi L(0)\ar@{-}[d]\\
L(2)\ar@{-}[dr]&& \Pi L(1)\ar@{-}[dl]\\
&L(0)&
}\\
\xymatrix@-1.5pc{&P(1)&\\
&L(1)\ar@{-}[d]\ar@{-}[d]&\\
&\Pi L(0)\ar@{-}[d]&\\
&L(2)\ar@{-}[d]&\\
&L(0)\ar@{-}[d]&\\
&L(1)&\\
}
\qquad
\xymatrix@-1.5pc{&P(2)&\\
&L(2)\ar@{-}[dl]\ar@{-}[dr]&\\
L(0)\ar@{-}[d]&&L(3)\ar@{-}[dddl]\\
L(1)\ar@{-}[d]&&\\
\Pi L(0)\ar@{-}[dr]&&\\
&L(2)&}
 \qquad
\xymatrix@-1.5pc{&P(a)&\\
&L(a)\ar@{-}[dl]\ar@{-}[dr]&\\
L(a+1)\ar@{-}[dr]&&L(a-1)\ar@{-}[dl]\\
&L(a)&
}
\end{gather*}

\subsection[g=q(3)]{$\boldsymbol{\g=\q(3)}$} The radical filtrations are $(a\geq 3)$:

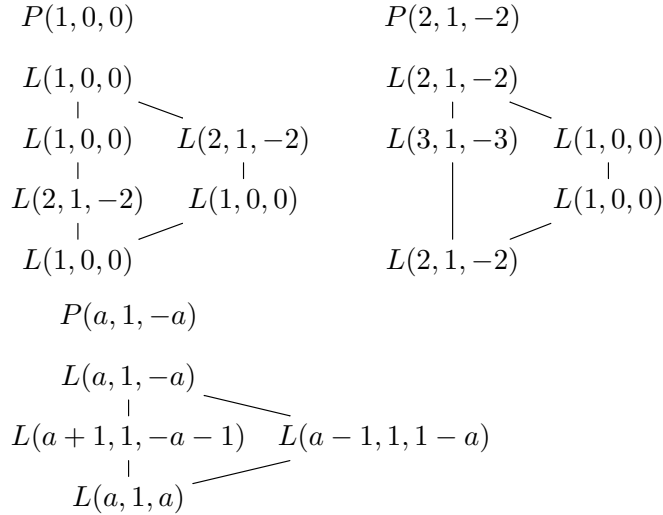
\begin{gather*}
\xymatrix@-1.5pc{P(1,0,0)&\\
L(1,0,0)\ar@{-}[dr]\ar@{-}[d]&\\
 L(1,0,0)\ar@{-}[d]&L(2,1,-2)\ar@{-}[d]\\
 L(2,1,-2)\ar@{-}[d]&L(1,0,0)\ar@{-}[dl]\\
L(1,0,0)&
}\qquad
\xymatrix@-1.5pc{P(2,1,-2)&\\
L(2,1,-2)\ar@{-}[d]\ar@{-}[dr]&\\
L(3,1,-3)\ar@{-}[dd]&L(1,0,0)\ar@{-}[d]\\
&L(1,0,0)\ar@{-}[dl]\\
L(2,1,-2)&
}
\\
\xymatrix@-1.5pc{P(a,1,-a)&\\
L(a,1,-a)\ar@{-}[d]\ar@{-}[dr]&\\
L(a+1,1,-a-1)\ar@{-}[d]&L(a-1,1,1-a)\ar@{-}[dl]\\
L(a,1,a)&\\
}
 \end{gather*}

The following radical filtrations are deduced from the fact that $\theta\colon P(a)\rightarrow\Pi P(a)$ corresponds to $\text{id}\colon P_{\mathfrak{sq}}(a)\rightarrow P_{\mathfrak{sq}}(a)$ as seen from Lemma~\ref{q relations}:
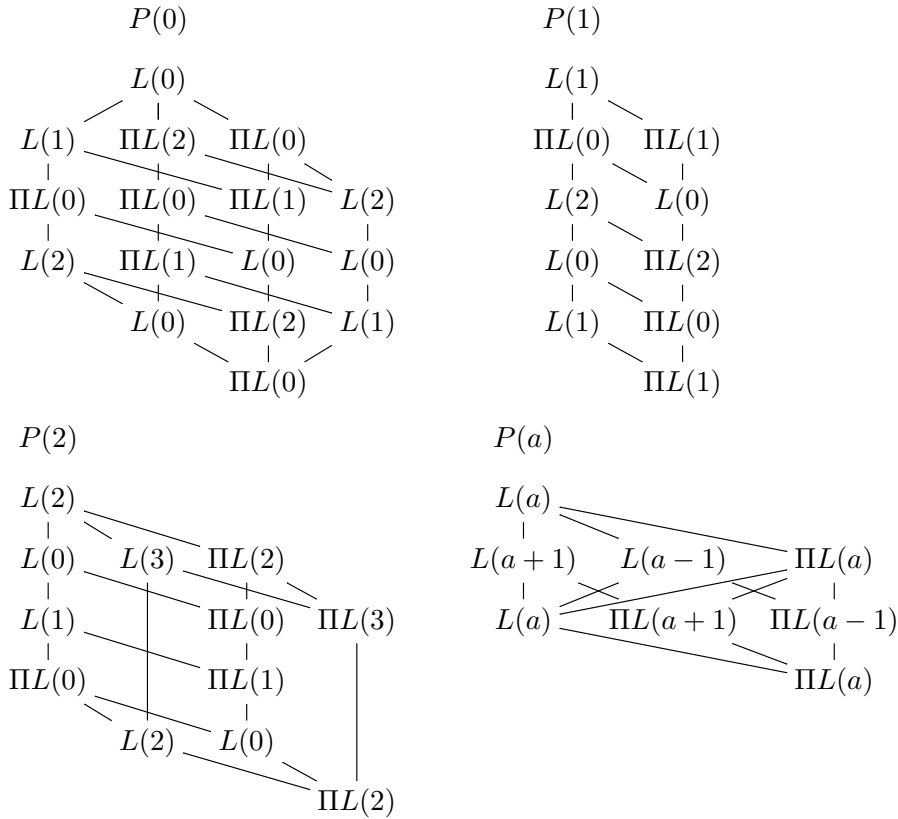
\begin{gather*}
 \xymatrix@-1.5pc{&P(0)&\\
 &L(0)\ar@{-}[dl]\ar@{-}[d]\ar@{-}[d]\ar@{-}[dr]&&\\
 L(1)\ar@{-}[d]\ar@{-}[drr]&\Pi L(2)\ar@{-}[d]\ar@{-}[drr]	&\Pi L(0)\ar@{-}[d]\ar@{-}[dr]\\
\Pi L(0)\ar@{-}[d]\ar@{-}[drr]&\Pi L(0)\ar@{-}[d]\ar@{-}[drr]	&\Pi L(1)\ar@{-}[d]& L(2)\ar@{-}[d]\\
L(2)\ar@{-}[drr]\ar@{-}[dr]& \Pi L(1)\ar@{-}[d]\ar@{-}[drr]	&L(0)\ar@{-}[d] & L(0)\ar@{-}[d]\\
&L(0)\ar@{-}[dr]&	\Pi L(2)\ar@{-}[d]& L(1)\ar@{-}[dl]\\
&&\Pi L(0)&&
}
\qquad
\xymatrix@-1.5pc{&P(1)&&\\
&L(1)\ar@{-}[d]\ar@{-}[d]\ar@{-}[dr]&&\\
&\Pi L(0)\ar@{-}[d]\ar@{-}[dr]&\Pi L(1)\ar@{-}[d]&\\
&L(2)\ar@{-}[d]\ar@{-}[dr]&L(0)\ar@{-}[d]&\\
&L(0)\ar@{-}[d]\ar@{-}[dr]&\Pi L(2)\ar@{-}[d]\\
&L(1)\ar@{-}[dr]&\Pi L(0)\ar@{-}[d]\\
&&\Pi L(1)&
}
\\
\xymatrix@-1.5pc{P(2)&\\
L(2)\ar@{-}[d]\ar@{-}[dr]\ar@{-}[drr]&\\
L(0)\ar@{-}[drr]\ar@{-}[d]&L(3)\ar@{-}[ddd]\ar@{-}[drr]		&\Pi L(2)\ar@{-}[d]\ar@{-}[dr]\\
L(1)\ar@{-}[d]\ar@{-}[drr]&&		\Pi L(0)\ar@{-}[d]&\Pi L(3)\ar@{-}[ddd]\\
\Pi L(0)\ar@{-}[dr]\ar@{-}[drr]&&		\Pi L(1)\ar@{-}[d]\\
&L(2)\ar@{-}[drr]&		L(0)\ar@{-}[dr]\\
&&&\Pi L(2)
}
 \qquad
\xymatrix@-1.5pc{P(a)&\\
L(a)\ar@{-}[d]\ar@{-}[dr]\ar@{-}[drr]\\
L(a+1)\ar@{-}[d]\ar@{-}[dr]&L(a-1)\ar@{-}[dl]\ar@{-}[dr]	&\Pi L(a)\ar@{-}[dl]\ar@{-}[d]\ar@{-}[dll]\\
L(a)\ar@{-}[drr]&\Pi L(a+1)\ar@{-}[dr]&\Pi L(a-1)\ar@{-}[d]\\
&&\Pi L(a)
}
\end{gather*}

\subsection*{Acknowledgements}
The authors would like to thank Dimitar Grantcharov for numerous helpful discussions. N.G.\ was supported by NSF grant DGE 1746045 and V.S.\ was supported by NSF grant 1701532.

\pdfbookmark[1]{References}{ref}
\LastPageEnding


\begin{thebibliography}{99}
\footnotesize\itemsep=0pt

\bibitem{Ben}
Benson D.J., Representations and cohomology. {I}.~Basic representation theory
 of finite groups and associative algebras, \textit{Cambridge Studies in
 Advanced Mathematics}, Vol.~30, Cambridge University Press, Cambridge, 1991.

\bibitem{BKN}
Boe B.D., Kujawa J.R., Nakano D.K., Cohomology and support varieties for {L}ie
 superalgebras, \href{https://doi.org/10.1090/S0002-9947-2010-05096-2}{\textit{Trans. Amer. Math. Soc.}} \textbf{362} (2010),
 6551--6590, \href{https://arxiv.org/abs/math.RT/0609363}{arXiv:math.RT/0609363}.

\bibitem{B1}
Brundan J., Kazhdan--{L}usztig polynomials and character formulae for the {L}ie
 superalgebra {${\mathfrak q}(n)$}, \href{https://doi.org/10.1016/S0001-8708(03)00073-2}{\textit{Adv. Math.}} \textbf{182} (2004),
 28--77, \href{https://arxiv.org/abs/math.RT/0207024}{arXiv:math.RT/0207024}.

\bibitem{B2}
Brundan J., Modular representations of the supergroup {$Q(n)$}.~{II},
 \href{https://doi.org/10.2140/pjm.2006.224.65}{\textit{Pacific~J. Math.}} \textbf{224} (2006), 65--90.

\bibitem{BD}
Brundan J., Davidson N., Type {C} blocks of super category {$\mathcal O$},
 \href{https://doi.org/10.1007/s00209-019-02400-y}{\textit{Math.~Z.}} \textbf{293} (2019), 867--901, \href{https://arxiv.org/abs/1702.05055}{arXiv:1702.05055}.

\bibitem{CK}
Cheng S.-J., Kwon J.-H., Finite-dimensional half-integer weight modules over
 queer {L}ie superalgebras, \href{https://doi.org/10.1007/s00220-015-2544-0}{\textit{Comm. Math. Phys.}} \textbf{346} (2016),
 945--965, \href{https://arxiv.org/abs/1505.06602}{arXiv:1505.06602}.

\bibitem{CKW}
Cheng S.-J., Kwon J.-H., Wang W., Character formulae for queer {L}ie
 superalgebras and canonical bases of types {$A/C$}, \href{https://doi.org/10.1007/s00220-016-2809-2}{\textit{Comm. Math.
 Phys.}} \textbf{352} (2017), 1091--1119, \href{https://arxiv.org/abs/1512.00116}{arXiv:1512.00116}.

\bibitem{CW}
Cheng S.-J., Wang W., Dualities and representations of {L}ie superalgebras,
 \textit{Graduate Studies in Mathematics}, Vol.~144, \href{https://doi.org/10.1090/gsm/144}{Amer. Math. Soc.},
 Providence, RI, 2012.

\bibitem{Erd}
Erdmann K., Blocks of tame representation type and related algebras,
 \textit{Lecture Notes in Math.}, Vol.~1428, \href{https://doi.org/10.1007/BFb0084003}{Springer-Verlag}, Berlin, 1990.

\bibitem{Fri}
Frisk A., Typical blocks of the category {$\mathcal O$} for the queer {L}ie
 superalgebra, \href{https://doi.org/10.1142/S0219498807002417}{\textit{J.~Algebra Appl.}} \textbf{6} (2007), 731--778.

\bibitem{Ger}
Germoni J., Indecomposable representations of special linear {L}ie
 superalgebras, \href{https://doi.org/10.1006/jabr.1998.7520}{\textit{J.~Algebra}} \textbf{209} (1998), 367--401.

\bibitem{GS}
Gruson C., Serganova V., Cohomology of generalized supergrassmannians and
 character formulae for basic classical {L}ie superalgebras, \href{https://doi.org/10.1112/plms/pdq014}{\textit{Proc.
 Lond. Math. Soc.}} \textbf{101} (2010), 852--892, \href{https://arxiv.org/abs/0906.0918}{arXiv:0906.0918}.

\bibitem{GS2}
Gruson C., Serganova V., Bernstein--{G}elfand--{G}elfand reciprocity and
 indecomposable projective modules for classical algebraic supergroups,
 \href{https://doi.org/10.17323/1609-4514-2013-13-2-281-313}{\textit{Mosc. Math.~J.}} \textbf{13} (2013), 281--313, \href{https://arxiv.org/abs/1111.6959}{arXiv:1111.6959}.

\bibitem{GS3}
Gruson C., Serganova V., A journey through representation theory: from finite
 groups to quivers via algebras, \textit{Universitext}, \href{https://doi.org/10.1007/978-3-319-98271-7}{Springer}, Cham, 2018.

\bibitem{Jan}
Jantzen J.C., Representations of algebraic groups, 2nd ed., \textit{Mathematical Surveys
 and Monographs}, Vol.~107, Amer. Math. Soc., Providence, RI, 2003.

\bibitem{Kac}
Kac V.G., Lie superalgebras, \href{https://doi.org/10.1016/0001-8708(77)90017-2}{\textit{Adv. Math.}} \textbf{26} (1977), 8--96.

\bibitem{Man}
Manin Yu.I., Gauge field theory and complex geometry, \textit{Grundlehren der
 Mathematischen Wissenschaften}, Vol. 289, \href{https://doi.org/10.1007/978-3-662-07386-5}{Springer-Verlag}, Berlin, 1997.

\bibitem{Maz}
Mazorchuk V., Miemietz V., Serre functors for {L}ie algebras and superalgebras,
 \href{https://doi.org/10.5802/aif.2698}{\textit{Ann. Inst. Fourier (Grenoble)}} \textbf{62} (2012), 47--75,
 \href{https://arxiv.org/abs/1008.1166}{arXiv:1008.1166}.

\bibitem{Mein}
Meinrenken E., Clifford algebras and {L}ie theory, \textit{Ergebnisse der
 Mathematik und ihrer Grenzgebiete. 3.~Folge. A Series of Modern Surveys in
 Mathematics}, Vol.~58, \href{https://doi.org/10.1007/978-3-642-36216-3}{Springer}, Heidelberg, 2013.

\bibitem{P}
Penkov I., Characters of typical irreducible finite-dimensional {${\mathfrak
 q}(n)$}-modules, \href{https://doi.org/10.1007/BF01077312}{\textit{Funct. Anal. Appl.}} \textbf{20} (1986), 30--37.

\bibitem{PS2}
Penkov I., Serganova V., Characters of finite-dimensional irreducible
 {${\mathfrak q}(n)$}-modules, \href{https://doi.org/10.1023/A:1007367827082}{\textit{Lett. Math. Phys.}} \textbf{40} (1997),
 147--158.

\bibitem{PS1}
Penkov I., Serganova V., Characters of irreducible {$G$}-modules and cohomology
 of {$G/P$} for the {L}ie supergroup {$G=Q(N)$}, \href{https://doi.org/10.1007/BF02399196}{\textit{J.~Math. Sci.}} \textbf{84} (1997), 1382--1412.

\bibitem{Ser1}
Serganova V., Quasireductive supergroups, in New Developments in {L}ie Theory
 and its Applications, \textit{Contemp. Math.}, Vol.~544, \href{https://doi.org/10.1090/conm/544/10753}{Amer. Math. Soc.},
 Providence, RI, 2011, 141--159.

\bibitem{Ser-ICM}
Serganova V., Finite dimensional representations of algebraic supergroups, in
 Proceedings of the {I}nternational {C}ongress of {M}athematicians~-- {S}eoul
 2014, {V}ol.~1, Kyung Moon Sa, Seoul, 2014, 603--632.

\bibitem{Serg}
Sergeev A.N., The centre of enveloping algebra for {L}ie superalgebra
 {$Q(n,{\mathbb C})$}, \href{https://doi.org/10.1007/BF00400431}{\textit{Lett. Math. Phys.}} \textbf{7} (1983), 177--179.

\end{thebibliography}
\end{document}